\documentclass[final,leqno]{siamltex704}
\usepackage{epsfig}
\usepackage{amsmath}
\usepackage{amssymb}
\usepackage{tikz}
\usepackage{graphicx}
\usepackage[notcite,notref]{showkeys}
\newtheorem{algorithm}{Weak Galerkin Algorithm}

\newtheorem{remark}{Remark}[section]
\newcommand{\bq}{{\bf q}}
\newcommand{\bn}{{\bf n}}
\newcommand{\bx}{{\bf x}}

\def\T{{\mathcal T}}
\def\E{{\mathcal E}}

\def\l{{\langle}}
\def\r{{\rangle}}

\def\bn{{\bf n}}
\def\bq{{\bf q}}

\def\pa{\partial}
\def\he{{\widehat e}}
\def\hQ{{\widehat Q}}
\def\hw{{\widehat w}}
\def\hs{{\hat s}}
\def\O{\Omega}

\def\bbQ{\mathbb{Q}}
\newcommand{\pT}{{\partial T}}

\def\3bar{{|\hspace{-.02in}|\hspace{-.02in}|}}
\renewcommand{\ldots}{\dotsc}
\title{Curved Elements in Weak Galerkin Finite Element Methods}

\author{
Dan Li\thanks{ School of Mathematical Sciences,  Nanjing Normal University, Nanjing
210023, China (danlimath@163.com).} \and
Chunmei Wang \thanks{Department of Mathematics, University of Florida, Gainesville, FL 32611 (chunmei.wang@ufl.edu). The research of Chunmei Wang was partially supported by National Science Foundation Grants DMS-2136380 and DMS-2206332.}
\and
Junping Wang\thanks{Division of Mathematical Sciences, National Science Foundation, Alexandria, VA 22314(jwang@nsf.gov). The research of Junping Wang was supported by the NSF IR/D program, while working at National Science Foundation. However, any opinion, finding, and conclusions or recommendations expressed in this material are those of the author and do not necessarily reflect the views of the National Science Foundation.}}

\begin{document}

\maketitle

\begin{abstract}
A mathematical analysis is established for the weak Galerkin finite element methods for the Poisson equation with Dirichlet boundary value when the curved elements are involved on the interior edges of the finite element partition or/and on the boundary of the whole domain in two dimensions. The optimal orders of error estimates for the weak Galerkin approximations in both the $H^1$-norm and the $L^2$-norm are established. Numerical results are reported to demonstrate the performance of the weak Galerkin methods on general curved polygonal partitions.
\end{abstract}


\pagestyle{myheadings}

\section{Introduction}\label{Section:Introduction}
We are concerned with the new developments of finite element methods for solving the Poisson equation by using the weak Galerkin (WG) finite element methods on the curved polygonal finite element partitions.

When the finite element methods are employed to solve the partial differential equation (PDE) problems, one of the steps is to partition the whole domain describing the original body or structure into finite elements (e.g., triangles, rectangles, etc.). The curved elements, a natural generalization of the polygonal elements, are applied for solving boundary value problems in the two-dimensional domain with an arbitrary/curved boundary. Although the engineers who conceive the matrix of finite element methods have used the curved finite elements for several decades, more research work needs to be done from the theoretical/mathematical point of view with regards to the error estimates of the numerical solution when the curved elements are concerned on the curved boundary of the domain and/or on the interior edges of the curved finite element partition. From the computational point of view, the curved elements make it possible to construct the finite-dimensional space for trial functions which is the subspace of the energy space of the boundary value problems in arbitrary/curved domain in two dimensions.

It is well-known that the numerical solutions of PDE problems with the curved boundaries by using the finite element methods may not be accurate \cite{b1961,b1971}. From a geometrical point of view, it is simple to replace the curved boundary by a polygon. The number of straight line segments can be increased until a desired geometrical accuracy is obtained. However, the geometrical accuracy may not always indicate the accuracy of the numerical approximation. Even if the piecewise polynomials of a higher degree are applied in the numerical scheme, the same accuracy may not be retained along the curved part of the original boundary as inside the domain or along the polygonal part of the boundary.  This behavior is known as the ``Babu$\breve{s}$ka Paradox''  in the literature \cite{b1961,b1971}. The numerical evidence was given in \cite{z1973} where the curved elements were proposed to be used along the curved part of the boundary.

The problem of accuracy of a finite element solution, near a curved boundary, has been investigated for several decades and some successful methods have been proposed to overcome it. The curved elements constructed in \cite{z1973} were closely associated with isoparametric elements, which were first introduced by Irons \cite{i1970} and were well-known in the technical literature \cite{sf}. The numerical results given in \cite{z1973} were very promising and suggested that using them could arrive at the same order of accuracy as in the case when the original boundary is a polygon and the triangular elements are applied \cite{z1968,bz}. \cite{z1973siam} proposed a finite element method which was applied for solving second order elliptic boundary value problems in domains with an arbitrary boundary, and the error bounds for a model problem were derived. Reader are referred to more references \cite{b1961,b1971,b1969,cr1972,mc1973,rr1968,sf,z1968,z1970,z1973,z1973siam,z1974siam}.

Weak Galerkin finite element method is a newly-developed numerical technique for PDEs where the differential operators in the variational formulation are reconstructed/approximated by using a framework that mimics the theory of distributions for piecewise polynomials. The usual regularity of the approximating functions is compensated by carefully-designed stabilizers. This WG method has been investigated for solving numerous model PDEs; see a limited list of references and references therein \cite{li-wang,mwy,mwy-biharmonic,ww2,ww4,ww6,ww7,ww8,ww9,ww12,ww13,wy,wy3655}. The research results indicate that the WG method has shown its great potential as a powerful numerical tool/technique in scientific computing. The fundamental difference between the WG methods and other existing finite element methods is the use of weak derivatives and weak continuities in the design of numerical schemes based on conventional weak forms for the underlying PDE problems. Due to its great structural flexibility, WG methods are well suited to a wide class of PDEs by providing the needed stability and accuracy in approximations. A recent development of WG, named ``Primal-Dual Weak Galerkin (PD-WG)'' has been proposed for problems for which the usual numerical methods are difficult to apply \cite{
 lwwhyper,cwwdivcurl,wcauchy,wmpdwg,wzcondif,cwcondif,wwfp,wwtrans,wcauchy2,wwcauchy,wwnondiv,lwspdwg,cwwinterface, cwwlp}. The essential idea of PD-WG is to interpret the numerical solutions as a constrained minimization of some functionals with constraints that mimic the weak formulation of the PDEs by using weak derivatives. The resulting Euler-Lagrange equation offers a symmetric scheme involving both the primal variable and the dual variable (Lagrange multiplier).

In the WG framework, the weak functions for second order elliptic equations possess the form of $v=\{v_0,v_b\}$ with $v=v_0$ representing the value of $v$ in the interior of each element and $v=v_b$ for the information of $v$ on the boundary of the element. Both $v_0$ and $v_b$ are approximated by polynomials of suitably-chosen degrees in the numerical approximation. To our best knowledge, all the existing results on WG were developed for finite element partitions with flat/straight sides. As most of the application problems involve physical domains with non-flat interfaces or boundaries, there is a great need of study for the WG method on curved elements.

For simplicity, we shall demonstrate the WG method on curved elements by using the Poisson equation with Dirichlet boundary condition. The model problem then seeks an unknown function $u\in H^1(\Omega)$ satisfying
\begin{eqnarray}
-\Delta u&=&f,\quad \mbox{in}\;\Omega,\label{pde}\\
        u&=&g,\quad\mbox{on}\;\partial\Omega,\label{bc}
\end{eqnarray}
where $\Omega$ is an open bounded domain in $\mathbb{R}^2$ with piecewise smooth and curved boundary $\partial\Omega$, and $\Delta=\nabla\cdot\nabla$ is the Laplacian operator with $\nabla u$ being the usual gradient operator.

The weak formulation of the second order elliptic model problem is as follows: Find $u\in H^1(\Omega)$ satisfying $u=g$ on $\partial\Omega$, such that
\begin{equation}\label{weakform}
(\nabla u,\nabla v)=(f,v), \qquad\forall v\in V,
\end{equation}
where $V=\{v\in H^1(\Omega), v=0\ \text{on}\ \partial\Omega\}$.

In this paper, the curved edges are assumed to appear on the interior interfaces of the partition and/or on the boundary of the whole domain in the analysis of the $H^1$-norm error estimate for the WG solution. For the simplicity of analysis, when it comes to the $L^2$-norm error estimate for the WG approximation, the curved edges are assumed to appear only on the boundary of the whole domain while the interior edges of the finite element partition are assumed to be straight line segments.

The paper is organized as follows. In Section \ref{Section:weak-gradient}, we shall review the definition of the weak gradient operator and its discrete analogue. In Section \ref{Section:CurvedElement}, we describe some properties for curved finite element partitions. In Section \ref{Section:wg-fem}, we shall state a weak Galerkin finite element scheme. Section \ref{Section:EU} is devoted to a discussion of the solution existence and uniqueness for the discrete system. In Section \ref{Section:ErrorEquation}, an error equation is derived. In Section
\ref{Section:Estimates}, we present some technical estimates for the usual $L^2$ projection operators. In Section
\ref{Section:H1ErrorEstimate}, we derive some optimal order error estimates for the WG approximations in both $H^1$ and $L^2$ norms. Section \ref{Section:NC} provides a new technique for calculating the integrations on curved polygons. Finally in Section \ref{Section:NE}, we conduct some numerical experiments for verifying the developed theories.

\section{Weak Gradient and Discrete Weak Gradient}\label{Section:weak-gradient}

The gradient operator is the differential operator used in the weak formulation (\ref{weakform}) of the second order model equation (\ref{pde})-(\ref{bc}). This section will review the weak gradient operator as well as its discrete version \cite{wy,wy3655}.

Let $T$ be a bounded domain with Lipschitz continuous boundary $\partial T$. By a {\em weak function} on $T$ we mean a function bundled with two or more components; each component represents a specific aspect of the function. In the interest of the gradient operator, we consider the weak function $v=\{v_0,v_b\}$ with two components $v_0\in L^2(T)$ and $v_b\in L^2(\partial T)$. The first component $v_0$ represents the value of $v$ in the interior of $T$, and the second one $v_b$ carries the value of $v$ on the boundary $\partial T$. Note that $v_b$ in general is not the trace of $v_0$ on $\partial T$, though taking the trace of $v_0$ on $\partial T$ is a viable option for $v_b$. Denote by $W(T)$ the space of all weak functions on $T$; i.e.,
$$
W(T)=\{v=\{v_0,v_b\}:v_0\in L^2(T),v_b\in L^2(\partial T)\}.
$$

For any $v\in W(T)$, the {\em weak gradient} of $v$ is defined as a bounded linear functional $\nabla_wv$ on $[H^1(T)]^2$ so that its action on each $\bq\in [H^1(T)]^2$ is given by
\begin{equation}\label{weak-gradient}
\langle\nabla_wv,\bq\rangle_T=-(v_0,\nabla\cdot\bq)_T+\langle v_b,\bq\cdot\bn\rangle_{\partial T},
\end{equation}
where $\bn$ is the outward normal direction on $\partial T$, $(v_0,\nabla\cdot\bq)_T=\int_Tv_0(\nabla\cdot\bq)dT$ is the inner product of $v_0$ and $\nabla\cdot\bq$ in $L^2(T)$, and $\langle v_b,\bq\cdot\bn\rangle_{\partial T}$ is the inner product of $v_b$ and $\bq\cdot\bn$ in $L^2(\partial T)$.

Denote by $P_r(T)$ the space of all polynomials on $T$ with total degree no more than $r$. A {\em discrete weak gradient} for $v=\{v_0,v_b\}$, denoted by $\nabla_{w,r}v$, is defined as an approximation of $\nabla_wv$ in the vector polynomial space $[P_r(T)]^2$ satisfying
\begin{equation}\label{dwd}
(\nabla_{w,r}v,\bq)_T=-(v_0,\nabla\cdot\bq)_T+\langle v_b,\bq\cdot\bn\rangle_{\partial T},\qquad \forall\bq\in[P_r(T)]^2.
\end{equation}

Assume the first component $v_0$ of $v=\{v_0,v_b\}$ is sufficiently regular such that $v_0\in H^1(T)$. Applying the usual integration by parts to the first term on the right hand side of (\ref{dwd}), we arrive at
\begin{equation}\label{dwd-2}
(\nabla_{w,r}v,\bq)_T=(\nabla v_0,\bq)_T+\langle v_b-v_0,\bq\cdot\bn\rangle_{\partial T},\qquad \forall\bq\in [P_r(T)]^2.
\end{equation}

\begin{remark}
In practical computation/implementation, the integrals over $T$ and $\partial T$ must be computed by using some numerical integration formulas. We assume these integrals are evaluated exactly.
\end{remark}

\section{Finite Elements with Curved Edges}\label{Section:CurvedElement}

A polygon with curved edges (PCE) is a bounded connected polygonal region in ${\mathbb R}^2$ bounded by a finite number of curved or straight edges. A curved polygonal partition of the domain $\Omega\subset{\mathbb R}^2$, denoted by $\T_h$, is defined as a family of PCEs, denoted by $\{T_j,j=1,2, \ldots\}$, satisfying two properties: (1) $\bigcup_{j=1,2,\ldots}T_j=\Omega$; and (2) for any $i,j (i\neq j)$, $T_i\cap T_j$ is either empty, or a common edge, or the vertices of $T_i$ and $T_j$. Each partition cell $T_j\in\T_h (j=1,2,\ldots)$ is called a curved element. A curved polygonal partition with a finite number of curved elements is called a curved finite element partition of the domain $\Omega$.

Let $\T_h=\{T_j\}_{j=1,\ldots,N}$ be a curved finite element partition of the domain $\Omega$. Denote by $h_T$ the diameter of the element $T$, and $h=\max_{T}h_T$ the meshsize of the partition $\T_h$. Denote by $|T|$ the area of the element $T\in\T_h$. Denote by ${\cal E}_h$ the set of all edges in ${\cal T}_h$ such that each edge $e\in\E_h$ is either on the boundary of $\Omega$ or shared by two distinct elements. Denote by $\E^0_h=\E_h\setminus{\partial\Omega}$ the set of all interior edges; i.e., for each edge $e\in\E_h^0$, there are two elements $T_j$ and $T_i$ ($i \neq j$) such that $e\subset T_j\cap T_i$.  Denote by $|e|$ or $h_e$ the length of the edge $e\in\E_h$. Assume that each element $T\in\T_h$ is a closed and simply connected polygon (see Fig. \ref{fig:shape-regular-element}).

\begin{figure}[h!]
\begin{center}\textbf{}
\begin{tikzpicture}[rotate=10]
\coordinate (A) at (-2,-1); \coordinate (B) at (2.5,-1.5); \coordinate (C) at (3.5, 0.5); \coordinate (D) at (2,1);
\coordinate (E) at (2.2, 3.3); \coordinate (F) at (-1, 3); \coordinate (CC) at (0,0);  \coordinate (Ae) at (1.2,1);
\coordinate (xe) at (-1.8625,1.1); \coordinate (AFc) at (-1.75, 1.5); \coordinate (AFcLeft) at (-2.2, 1.75);
\draw node[left] at (xe) {$\bx_e$}; \draw node[right] at (Ae) {$A_e$}; \draw node[left] at (A) {B}; \draw node[below] at (B) {C};
\draw node[right] at (C) {D}; \draw node[below] at (D) {E}; \draw node[above] at (E) {F}; \draw node[above] at (F) {A};
\draw node[left] at (AFcLeft) {$\bn$}; \draw (A)--(B)--(C)--(D)--(E)--(F);
\draw[dashed](A)--(Ae); \draw[dashed](F)--(Ae); \draw[->] (Ae)--(xe); \draw[->] (AFc)--(AFcLeft);
\filldraw[black] (A) circle(0.05); \filldraw[black] (F) circle(0.05); \filldraw[black] (Ae) circle(0.05); \filldraw[black] (xe) circle(0.035);
\draw (A) to [out=115, in=250] (F);
\end{tikzpicture}
\caption{Depiction of a shape-regular polygonal element $ABCDEFA$.}
\label{fig:shape-regular-element}
\end{center}
\end{figure}
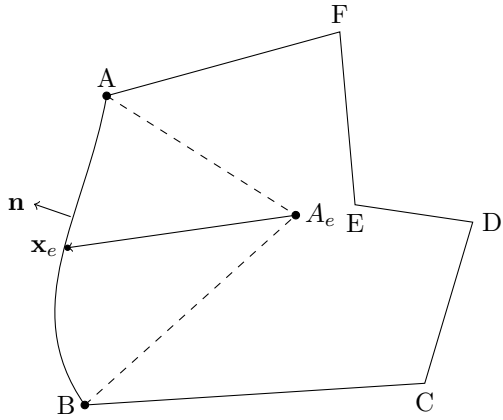
\medskip
\medskip

The curved finite element partition ${\cal T}_h$ is said to be shape regular if the conditions (A1)-(A4) are satisfied \cite{wy,mwy}.

\medskip
\begin{description}
\item[A1:] For each element $T\in\T_h$, there exists a positive constant $\varrho_v$ such that
\begin{equation*}\label{a1}
\varrho_v h_T^2\leq|T|.
\end{equation*}

\item[A2:] For each element $T\in\T_h$, there exist positive constants $\kappa$ and $\kappa^*$ such that
\begin{equation*}\label{a2}
\kappa h_T\leq h_e\leq\kappa^*h_T,
\end{equation*}
for each edge $e\subset\partial T$.

\item[A3:] For each element $T\in\T_h$ and each edge $e\subset\partial T$, there exists a ``pyramid'' $P(e,T,A_e)$ contained in $T$ such that its curved base is identical with $e$, its apex is $A_e\in T$, and its height is proportional to $h_T$ with a proportionality constant $\sigma_e$ bounded by a fixed positive number $\sigma^*$ from below. In other words, the height of the ``pyramid'' is given by $\sigma_eh_T$ such that $\sigma_e\ge\sigma^*>0$. The ``pyramid'' is also assumed to stand up above the curved base $e$ in the sense that the angle between the vector $\overrightarrow{A_e x_e}$, for any $x_e\in e$, and the outward normal direction of $e$ (i.e., the vector $\bn$ in Fig. \ref{fig:shape-regular-element}) is strictly acute by falling into an interval $[0,\theta_0]$ with $\theta_0<\frac{\pi}{2}$.

\item[A4:] For each element $T\in\T_h$, there is a simplex $S(T)$ circumscribed in $T$ that is shape regular and the diameter of $S(T)$, denoted by $h_{S(T)}$, is proportional to the diameter of $T$; i.e., $h_{S(T)}\leq\gamma_*h_T$ with a constant $\gamma_*$ independent of $T$. Furthermore, assume that each circumscribed simplex $S(T)$ intersects with only a fixed and small number of such simplices for all other elements $T\in\T_h$.
\end{description}

For the curved finite element partition $\T_h$, we assume that each curved edge can be straightened through a local mapping that is sufficiently smooth. More precisely, for each curved edge $e\subset\partial T,T\in\T_h$, assume that there exists a parametric representation
\[(x,y)=(\phi(\hat{s}),\psi(\hat{s})),\qquad\hat{s}\in\he=[0,h_e],\]
where $\phi=\phi(\hat{s})\in C^n$ and $\psi=\psi(\hat{s})\in C^n$ for some $n\ge1$, and at least one of the derivatives $\phi'(\hat{s})$ and $\psi'(\hat{s})$ is different from zero for $\hat{s}\in\he$. Assume that the mapping ${F}_e:=(\phi,\psi)$ from $\he$ to $e$ is globally invertible on the ``reference'' edge $\he$, and both ${F}_e$ and its inverse mapping $\widehat{F}_e:={F}_e^{-1}$ can be extended to the ``pyramid'' $P(e,T,A_e)$ as ${F}$ and $\widehat{F}:={F}^{-1}$; see Fig. \ref{fig:curved-element} for an illustration. We further assume that there exists a constant $C$ such that
\begin{equation}\label{EQ:mappingF}
\left|\frac{d^\alpha F}{d\hs^\alpha}\right|\leq C,
\end{equation}
for all $|\alpha|\le n$.

Let $e$ be a curved edge of the element $T\in\T_h$ with a parametric representation given by
\[\bx={F}_e(\hat{s}),\qquad\hat{s}\in[0,h_e],\]
where $\bx=(x,y)\in e$ and ${F}_e(\hat{s})=(\phi(\hat{s}),\psi(\hat{s}))$. With the mapping ${F}_e$ and its inverse $\widehat{F}_e:=F_e^{-1}$, any function $\hw\in L^2(\he)$ can be transformed into a function $w\in L^2(e)$ as follows
\begin{equation}\label{EQ:March08.001}
w(\bx):=\hw(\widehat{F}_e(\bx)),\qquad\bx\in e.
\end{equation}
Likewise, any function $w\in L^2(e)$ can be transformed into a function in $L^2(\he)$ by
\begin{equation}\label{EQ:March08.002}
\hw(\hat{s}):=w({F}_e(\hat{s})),\qquad\hat{s}\in\he.
\end{equation}
The relations (\ref{EQ:March08.001}) and (\ref{EQ:March08.002}) are
written respectively as
$$
w=\hw\circ\widehat{F}_e,\quad\hw=w\circ{F}_e.
$$

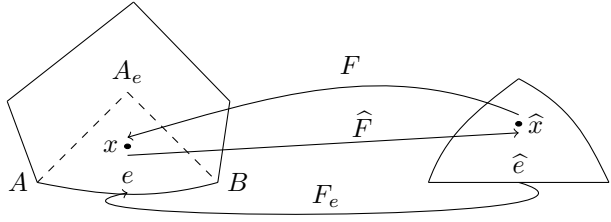
\begin{figure}
\begin{center}
\begin{tikzpicture}[xscale=0.8,yscale=0.6, rotate=0]
 \path (0,0) coordinate (A1); \path (3,0) coordinate (A2); \path (1.5,2) coordinate (A3); \path (-0.5, 1.8) coordinate (AC);
 \path (3.2, 1.8) coordinate (BC); \path (1.6, 4.0) coordinate (AD); \draw (A1) to  (AC); \draw (AC) to (AD); \draw (AD) to (BC);
 \draw (BC) to (A2); \draw (A1) to [out=-15, in=200] (A2); \draw[dashed] (A3) to  (A2); \draw[dashed] (A1) to  (A3);
 \path (1.5,-0.23) coordinate (Ecenter); \path (8,0.0) coordinate (Ehatcenter); \path (6.5,0) coordinate (A1n); \path (9.5,0) coordinate (A2n);
 \path (8,2.3) coordinate (A3n); \draw (A1n) to (A2n); \draw (A3n) to [out=-35, in=110] (A2n); \draw (A1n) to [out=75, in=-135] (A3n);
 \draw node[left] at (A1) {$A$}; \draw node[right] at (A2) {$B$}; \draw node[above] at (A3) {$A_e$}; \path (1.5,0.8) coordinate (LeftCenter);
 \path (2,1.6) coordinate (HatTC); \path (1.5,1.0) coordinate (LeftCenterUP); \path (1.5,0.6) coordinate (LeftCenterDown);
 \path (8,1.5) coordinate (RightCenterUP); \path (8,1.1) coordinate (RightCenterDown); \path (8,1.3) coordinate (RightCenter);
 \path (8.2,1.6) coordinate (TC);
 \filldraw[black] (LeftCenter) circle(0.05); \filldraw[black] (RightCenter) circle(0.05);
 \draw node[left] at (LeftCenter) {$x$}; \draw node[right] at (RightCenter) {$\widehat{x}$};
 \draw[<-] (LeftCenterUP) to [out=21, in=150] (RightCenterUP); \draw[<-] (RightCenterDown) to (LeftCenterDown);
 \draw[<-] (Ecenter) to [out=-165, in=-21] (Ehatcenter);
 \path (5.2,2.6) coordinate (CenterUp); \path (5.4,1.3) coordinate (CenterDown); \path (4.8,-0.3) coordinate (CenterDowne);
 \draw node at (CenterDown) {$\widehat{F}$}; \draw node at (CenterUp) {${F}$}; \draw node at (CenterDowne) {${F}_e$};
 \draw node[above] at (Ehatcenter) {$\widehat{e}$}; \draw node[above] at (Ecenter) {${e}$};
 \path (4,0) coordinate (A1n); \path (7,0) coordinate (A2n); \path (5.5,1.8) coordinate (A3n); \path (4.75,1.05) coordinate (A13n);
 \path (6.2,1.3) coordinate (A23n);
\end{tikzpicture}
\caption{Depiction of locally smooth mappings $F_e$ that straighten
curved edges for curved elements.} \label{fig:curved-element}
\end{center}
\end{figure}

\medskip

\section{Weak Galerkin Finite Element Schemes}\label{Section:wg-fem}

For any integer $\ell\ge0$, denote by ${P}_{\ell}(\he)$ the space of polynomials of degree $\ell$ on the straight reference edge $\he$. With the mapping $\widehat{F}_e:=F_e^{-1}$, the space of polynomials ${P}_{\ell}(\he)$ can be transformed into $V_b(e,\ell)={P}_{\ell}(\he)\circ\widehat{F}_e$ as follows:
$$
V_b(e,\ell)=\{\phi=\widehat{\phi}\circ\widehat{F}_e:\ \ \widehat{\phi}\in{P}_{\ell}(\he)\}.
$$
If the edge $e$ is a straight line segment, the mapping $F_e$ is required to be affine. Consequently, its inverse mapping $\widehat{F}_e$ is also affine so that $V_b(e,\ell)=P_{\ell}(e)$ is the usual space of polynomials of degree $\ell$ on $e$.

Let $k\ge 1$ be a given integer. On each element $T\in\T_h$, we define a local finite element space as
$$
W(k,T)=\{v=\{v_0,v_b\}:v_0\in P_k(T),v_b|_e\in V_b(e,k-1),\ \ \forall e\subset\partial T\}.
$$
By patching all the local finite element spaces $W(k,T)$ together with a common value $v_b$ on each interior edge in $\E_h^0$, we obtain a global finite element space, denoted by $W_h$; i.e.,
\begin{equation}\label{EQ:global-WFES}
W_h=\{v=\{v_0,v_b\}:\;v|_T\in W(k,T),\ v_b|_{\partial T_i\cap e}=v_b|_{\partial T_j\cap e},\ T\in\T_h,e\in\E_h^0\},
\end{equation}
where $v_b|_{\partial T_s\cap e}$ is the value of $v_b$ on the edge $e$ as seen from the element $T_s,\ s=i,j$. Denoted by $W^0_h$ a subspace of $W_h$ with vanishing value on $\partial\Omega$; i.e.,
\begin{equation}\label{vh0space}
W^0_h=\{v:\ v\in W_h,\ v_b=0~\mbox{on}~\partial\Omega\}.
\end{equation}

For each element $T\in\T_h$, denote by $Q_0$ the $L^2$ projection from $L^2(T)$ to $P_k(T)$. Denote by $\hQ_b$ the weighted $L^2$ projection from $L^2(\he)$ to ${P}_{k-1}(\he)$ with the corresponding Jacobian as the weight function. For each edge $e$, we define a projection operator $Q_b:L^2(e)\rightarrow V_b(e,k-1)$ as follows
$$
Q_bw\circ F_e:=\hQ_b(w\circ F_e),\qquad w\in L^2(e).
$$
Note that for the straight edge $e$, the operator $Q_b$ is easily seen to be the standard $L^2$ projection from $L^2(e)$ to $P_{k-1}(e)$. $Q_0$ and $Q_b$ collectively define a projection operator onto the weak finite element space $W_h$, denoted by
$$
Q_h=\{Q_0,Q_b\}.
$$
Denote by $\bbQ_h$ the $L^2$ projection from $[L^2(T)]^2$ onto $[P_{k-1}(T)]^2$.

For all $v,w\in W_h$, we introduce two bilinear forms as follows:
\begin{eqnarray*}
s(v,w)&=&\rho\sum_{T\in{\cal T}_h}h_T^{-1}\langle Q_bv_0-v_b,Q_bw_0-w_b\rangle_{\partial T},\\
a(v,w)&=&\sum_{T\in{\cal T}_h}(\nabla_wv,\nabla_ww)_T+s(v,w),
\end{eqnarray*}
where $\rho$ is any positive number of unit size. For simplicity, we shall take $\rho=1$ throughout the paper.

\begin{algorithm}
Find $u_h=\{u_0,u_b\}\in W_h$ satisfying $u_b=Q_bg$ on $\partial\Omega$, such that
\begin{equation}\label{wg}
a(u_h,v)=(f,v_0),\quad \forall v=\{v_0,v_b\}\in W_h^0.
\end{equation}
\end{algorithm}


\section{Existence and Uniqueness}\label{Section:EU}

The goal of this section is to examine the well-posedness of the weak Galerkin finite element scheme (\ref{wg}). Note that the bilinear form $a(\cdot,\cdot)$ is symmetric and non-negative in the space $W_h\times W_h$. Letting
\begin{equation}\label{3barnorm}
\3barv\3bar^2=a(v,v),
\end{equation}
we see that the functional $\3bar\cdot\3bar$ defines a semi-norm in $W_h$. Furthermore, the following result holds true.

\begin{lemma}\label{LEMMA:triplebarnorm}
The functional $\3bar\cdot\3bar$ given by (\ref{3barnorm}) defines a norm in the subspace $W_h^0$, provided that the meshsize $h$ is sufficiently small.
\end{lemma}

\begin{proof}
It suffices to check the positivity property for $\3bar\cdot\3bar$. Assume that $\3bar v\3bar=0$ for $v\in W_h^0$. It follows that
\[(\nabla_wv,\nabla_wv)+\sum_{T\in\T_h}h_T^{-1}\langle Q_bv_0-v_b,Q_bv_0-v_b\rangle_\pT=0,\]
which implies that $\nabla_wv=0$ on each element $T$ and $Q_bv_0=v_b$ on each $\pT$. It follows from $\nabla_wv=0$ and (\ref{dwd-2}) that for any $\bq\in[P_{k-1}(T)]^2$,
\begin{eqnarray*}
0&=&(\nabla_wv,\bq)_T\\
 &=&(\nabla v_0,\bq)_T+\langle v_b-v_0,\bq\cdot\bn\rangle_\pT\\
 &=&(\nabla v_0,\bq)_T+\langle Q_bv_0-v_0,\bq\cdot\bn\rangle_\pT,
\end{eqnarray*}
which leads to
\begin{equation}\label{EQ:July05:000}
(\nabla v_0,q)_T=\langle v_0-Q_bv_0,\bq\cdot\bn\rangle_\pT.
\end{equation}

On each straight edge $e\subset\pT$, since $Q_b$ is the usual $L^2$ projection onto the space $P_{k-1}(e)$ and $\bq\cdot\bn|_e\in P_{k-1}(e)$, then
\begin{equation}\label{EQ:July06-2016:001new}\nonumber
\langle v_0-Q_bv_0,\bq\cdot\bn\rangle_e=0.
\end{equation}
If $e\subset\pT$ is a curved edge, then the above identity is generally not valid. However, Lemma \ref{Lemma:Wonderful} can be used to show that there exists a constant $C$ such that
\begin{equation}\label{EQ:July05:005}
|\langle v_0-Q_bv_0,\bq\cdot\bn\rangle_e|\le Ch_e\|\nabla v_0\|_T \ \|\bq\|_{T}.
\end{equation}
By combining (\ref{EQ:July05:000}) with (\ref{EQ:July05:005}) we obtain
$$
|(\nabla v_0,\bq)_T|\leq Ch_e\|\nabla v_0\|_T \ \|\bq\|_{T},
$$
for all $\bq\in[P_{k-1}(T)]^2$. It follows that
$$
\|\nabla v_0\|_T\leq Ch_e\|\nabla v_0\|_T,
$$
which shows that $\nabla v_0=0$ for sufficiently small meshsize $h$. Thus, $v_0$ is a constant on each $T\in\T_h$ and hence $Q_b v_0$ is a constant on each $\partial T$. Using the fact that $Q_bv_0=v_b$ and $v_b=0$ on $\partial\Omega$, we have $v_0=0$ and $v_b=0$. This completes the proof of the lemma.
\end{proof}

\medskip
\begin{theorem}\label{THM:ExistenceUniqueness}
Assume that the curved finite element partition $\T_h$ is shape-regular with sufficiently small meshsize $h$. The weak Galerkin finite element scheme (\ref{wg}) has one and only one solution.
\end{theorem}

\smallskip

\begin{proof}
It suffices to prove the uniqueness. Assume that $u_h^{(1)}$ and $u_h^{(2)}$ are two different solutions of (\ref{wg}), then $\epsilon_h=u_h^{(1)}-u_h^{(2)}$ would satisfy
$$
a(\epsilon_h,v)=0,\qquad\forall v\in W_h^0.
$$
Note that $\epsilon_h\in W_h^0$. Letting $v=\epsilon_h$ in the above equation gives
$$
\3bar\epsilon_h\3bar^2=a(\epsilon_h,\epsilon_h)=0.
$$
It follows that $\epsilon_h\equiv 0$, or equivalently, $u_h^{(1)}\equiv u_h^{(2)}$. This completes the proof of the theorem.
\end{proof}
\medskip

\section{Error Equation}\label{Section:ErrorEquation}

We start this section by deriving a useful result for the discrete weak gradient operator.

\smallskip

\begin{lemma}
Let $Q_h$ and $\bbQ_h$ be the $L^2$ projection operators defined in the previous sections. On each element $T\in\T_h$, we have that for any $\phi\in H^1(T)$,
\begin{equation}\label{key}
(\nabla_wQ_h\phi,\tau)_T=(\nabla\phi,\tau)_T+\l Q_b\phi-\phi,\tau\cdot\bn\r_\pT ,\quad\forall \tau\in[P_{k-1}(T)]^2.
\end{equation}
Note that $\langle Q_b\phi-\phi,\tau\cdot\bn\rangle_{\partial T}\neq0$ when there is at least one curved segment on $\pT$ .
\end{lemma}

\begin{proof}
Using (\ref{dwd}), the integration by parts and the definitions of $Q_h$ and $\bbQ_h$, we have for any $\tau\in[P_{k-1}(T)]^2$ that
\begin{eqnarray*}
(\nabla_w(Q_h\phi),\tau)_T
&=&-(Q_0\phi,\nabla\cdot\tau)_T+\langle Q_b\phi,\tau\cdot\bn\rangle_{\pT}\\
&=&-(\phi,\nabla\cdot\tau)_T+\langle\phi,\tau\cdot\bn\rangle_{\partial T}+\langle Q_b\phi-\phi,\tau\cdot\bn\rangle_{\partial T}\\
&=&(\nabla\phi,\tau)_T+\langle Q_b\phi-\phi,\tau\cdot\bn\rangle_{\partial T},
\end{eqnarray*}
which implies the desired identity (\ref{key}).
\end{proof}

\begin{lemma}\label{Lemma:error-equation-prep}
For any $w\in H^1(\Omega)\cap H^{1+\gamma}(\Omega)$ with $\gamma>\frac12$, and $v\in W_h$, the following identity holds true
\begin{equation}\label{April-5:888}
\begin{split}
\sum_{T\in\T_h}(\nabla_wQ_hw,\nabla_wv)_T
=&\sum_{T\in\T_h}(\nabla w,\nabla v_0)_T+\sum_{T\in\T_h}\l\nabla w\cdot\bn,v_b-v_0\r_\pT \\
 &+\ell_1(w,v)+\ell_2(w,v),
\end{split}
\end{equation}
where
\begin{eqnarray*}
\ell_1(w,v)&=&\sum_{T\in\T_h}\langle(\nabla w-\bbQ_h\nabla w)\cdot\bn,v_0-v_b\rangle_\pT,\\
\ell_2(w,v)&=&\sum_{T\in\T_h}\l Q_bw-w,\nabla_wv\cdot\bn\r_\pT.
\end{eqnarray*}
\end{lemma}

\begin{proof}
From (\ref{key}) with $\phi=w$ and $\tau=\nabla_wv$ we obtain
\begin{equation}\label{April-5:001}
(\nabla_wQ_hw,\nabla_wv)_T=(\nabla w,\nabla_wv)_T+\l Q_bw-w,\nabla_wv\cdot\bn\r_\pT.
\end{equation}
Using $(\nabla w,\nabla_wv)_T=(\bbQ_h\nabla w,\nabla_wv)_T$,(\ref{April-5:001}) can be rewritten as
\begin{equation}\label{April-5:002}
(\nabla_wQ_hw,\nabla_wv)_T=(\bbQ_h\nabla w,\nabla_wv)_T+\l Q_bw-w,\nabla_wv\cdot\bn\r_\pT.
\end{equation}
Now by applying (\ref{dwd-2}) with $\bq=\bbQ_h\nabla w$ to the first term on the right-hand side of (\ref{April-5:002}) we arrive at
\begin{equation}\label{April-5:003}
\begin{split}
&(\nabla_wQ_hw,\nabla_wv)_T\\
=&(\bbQ_h\nabla w,\nabla v_0)_T+\l\bbQ_h\nabla w\cdot\bn,v_b-v_0\r_\pT+\l Q_bw-w,\nabla_wv\cdot\bn\r_\pT\\
=&(\nabla w,\nabla v_0)_T+\l\bbQ_h\nabla w\cdot\bn,v_b-v_0\r_\pT+\l Q_bw-w,\nabla_wv\cdot\bn\r_\pT\\
=&(\nabla w,\nabla v_0)_T+\l\nabla w\cdot\bn,v_b-v_0\r_\pT+\l(\bbQ_h\nabla w-\nabla w)\cdot\bn,v_b-v_0\r_\pT\\
 &+\l Q_bw-w,\nabla_wv\cdot\bn\r_\pT.
\end{split}
\end{equation}
Summing (\ref{April-5:003}) over all $T\in\T_h$ gives rise to
\begin{equation*}\label{April-5:004}
\begin{split}
\sum_{T\in\T_h}(\nabla_wQ_hw,\nabla_wv)_T
=&\sum_{T\in\T_h}(\nabla w,\nabla v_0)_T+\sum_{T\in\T_h}\l\nabla w\cdot\bn,v_b-v_0\r_\pT\\
 &+\ell_1(w,v)+\ell_2(w,v),
\end{split}
\end{equation*}
which confirms the identity (\ref{April-5:888}).
\end{proof}

\medskip

Let $u_h=\{u_0,u_b\}\in W_h$ be the weak Galerkin finite element solution of (\ref{wg}) and $u$ be the exact solution of (\ref{pde})-(\ref{bc}). By error function, denoted by $e_h$, we mean the difference of the $L^2$ projection of the exact solution $u$ and its weak Galerkin finite element solution $u_h$; i.e., $e_h=Q_h u-u_h=\{e_0,e_b\}$ with
$$
e_0=Q_0u-u_0,\quad e_b=Q_bu-u_b.
$$

We are ready to derive an error equation for the weak Galerkin finite element scheme (\ref{wg}) which the error function $e_h$ will satisfy.

\begin{theorem}\label{Theorem:error-equation} Assume that the exact solution $u$ of the model problem (\ref{pde})-(\ref{bc}) is sufficiently regular such that $u\in H^1(\Omega)\cap H^{1+\gamma}(\Omega)$, $\gamma>\frac12$. Let $e_h$ be the error function of the weak Galerkin finite element scheme (\ref{wg}). Then, for any $v\in W_h^0$ there holds
\begin{eqnarray}
a(e_h,v)=\ell_1(u,v)+\ell_2(u,v)+s(Q_hu,v).\label{ee}
\end{eqnarray}
\end{theorem}

\begin{proof}
By testing (\ref{pde}) with the first component $v_0$ of $v=\{v_0,v_b\}\in W_h^0$, we have
\begin{equation}\label{m1}
\sum_{T\in\T_h}(\nabla u,\nabla v_0)_T-\sum_{T\in\T_h}\langle\nabla u\cdot\bn,v_0-v_b\rangle_\pT=(f,v_0),
\end{equation}
where we have used the fact that $\sum_{T\in\T_h}\langle\nabla u\cdot\bn,v_b\rangle_\pT=0$ since $v_b=0$ on $\partial\Omega$. Next, from Lemma
\ref{Lemma:error-equation-prep} we obtain
\begin{equation}\label{April-5:110}
\begin{split}
\sum_{T\in\T_h}(\nabla_wQ_hu,\nabla_wv)_T
=&\sum_{T\in\T_h}(\nabla u,\nabla v_0)_T+\sum_{T\in\T_h}\l\nabla u\cdot\bn,v_b-v_0\r_\pT\\
 &+\ell_1(u,v)+\ell_2(u,v).
\end{split}
\end{equation}
Combining (\ref{m1}) with (\ref{April-5:110}) yields
\begin{eqnarray*}
\sum_{T\in\T_h}(\nabla_wQ_hu,\nabla_wv)_T=(f,v_0)+\ell_1(u,v)+\ell_2(u,v).
\end{eqnarray*}
Adding $s(Q_hu,v)$ to both sides of the above equation gives
\begin{equation}\label{j2}
a(Q_hu,v)=(f,v_0)+\ell_1(u,v)+\ell_2(u,v)+s(Q_hu,v).
\end{equation}
Finally, subtracting (\ref{wg}) from (\ref{j2}) yields
\begin{eqnarray*}
a(e_h,v)=\ell_1(u,v)+\ell_2(u,v)+s(Q_hu,v),\quad \forall v\in W_h^0,
\end{eqnarray*}
which completes the proof of the lemma.
\end{proof}

\section{Some Technical Estimates}\label{Section:Estimates}

For any function $\varphi\in H^1(T)$, we use the ideas presented in \cite{wy3655} to obtain the following trace inequality
\begin{equation}\label{trace}
\|\varphi\|_e^2\leq C\left(h_T^{-1}\|\varphi\|_T^2+h_T\|\nabla\varphi\|_{T}^2\right).
\end{equation}
If $\varphi$ is a polynomial, using the inverse inequality, the trace inequality  (\ref{trace}) becomes
\begin{equation}\label{trace2}
\|\varphi\|_e^2\leq Ch_T^{-1}\|\varphi\|_T^2.
\end{equation}

In the weak finite element space $W_h$, we introduce the following discrete $H^1$-seminorm; i.e.,
\begin{equation}\label{EQ:discrete-H1-norm}
\|v\|_{1,h}=\left(\sum_{T\in\T_h}\|\nabla v_0\|_T^2+h_T^{-1}\|Q_bv_0-v_b\|_\pT^2\right)^{\frac12},\forall v=\{v_0,v_b\}\in W_h.
\end{equation}
It is not hard to see that $\|\cdot\|_{1,h}$ indeed provides a norm for the subspace $W_h^0$ which consists of the weak finite element functions with vanishing boundary value.

\begin{lemma}\label{Lemma:Wonderful}
On each element $T\in\T_h$, for any $\phi\in H^1(T)$ and $\bq\in[P_{k-1}(T)]^2$, there exists a positive constant $C$, such that
\begin{equation}\label{EQ:July06-2016:600}
|\langle\phi-Q_b\phi,\bq\cdot\bn\rangle_e|\le\left\{
\begin{array}{ll}
Ch_e^{1/2}\|\phi-Q_b\phi\|_\pT \,\|\bq\|_{T},\qquad\qquad\qquad \mbox{for $k\ge 1$,}\\
Ch_e^{3/2}\|\phi-Q_b\phi\|_\pT \,(\|\bq\|_{T}+\|\nabla \bq\|_T),\quad\mbox{for $k\ge 2$}.
\end{array}
\right.
\end{equation}
Moreover, taking $\phi=v_0\in P_k(T)$, then for any $k\ge1$, there holds
\begin{equation}\label{EQ:July06-2016:800}
|\langle v_0-Q_bv_0,\bq\cdot\bn\rangle_e|\le Ch_e\|\nabla v_0\|_T \ \|\bq\|_{T}.
\end{equation}
\end{lemma}

\begin{proof}
For each straight edge $e\subset \pT$, we have
\begin{equation}\label{EQ:July06-2016:001}
\langle\phi-Q_b\phi,\bq\cdot\bn\rangle_e=0,
\end{equation}
since $Q_b$ is the usual $L^2$ projection onto the space $P_{k-1}(e)$ and $\bq\cdot\bn|_e\in P_{k-1}(e)$. If $e\subset\pT$ is a curved edge, then the identity (\ref{EQ:July06-2016:001}) generally does not hold true. However, by mapping to the reference edge $\he$, we have
\begin{eqnarray*}
\langle\phi-Q_b\phi,\bq\cdot\bn\rangle_e
&=&\int_{\he}({\widehat\phi}-{\widehat Q}_b{\widehat\phi}){\widehat\bq}\cdot\bn |J_e|d\he\\
&=&\int_{\he}({\widehat\phi}-{\widehat Q}_b{\widehat\phi})({\widehat\bq}\cdot\bn-\chi)|J_e|d\he,
\end{eqnarray*}
where $J_e$ is the Jacobian of the mapping, and $\chi \in P_{k-1}(\he)$ is any polynomial of degree $k-1$ on the reference edge $\he$. Thus, using the Cauchy-Schwarz inequality gives
\begin{equation}\label{EQ:July05:001}
|\langle\phi-Q_b\phi,\bq\cdot\bn\rangle_e|
\le Ch_e^{k}\left(\int_{\he}\left|{\widehat\phi}-{\widehat Q}_b{\widehat\phi}\right|^2|J_e|d\he\right)^{\frac12}
      \left(\int_{\he}\left|\frac{d^{k}({\widehat\bq\cdot\bn})}{d\hs^k}\right|^2|J_e| d\he\right)^{\frac12}.
\end{equation}
From the chain rule and the assumption (\ref{EQ:mappingF}) we have
\begin{eqnarray*}
\left|\frac{d^{k}({\widehat\bq}\cdot\bn)}{d\hs^k}\right|
\leq\sum_{|\alpha|=0}^kC_\alpha|\nabla^\alpha\bq|
\leq\sum_{|\alpha|=0}^{k-1}C_\alpha|\nabla^\alpha\bq|,
\end{eqnarray*}
where we have used the fact that $\nabla^k\bq=0$ as $\bq$ is a polynomial of degree $k-1$. By mapping back to the edge $e$ we have
\begin{equation}\label{EQ:July05:003}
\int_{\he}\left|\frac{d^{k}({\widehat\bq}\cdot\bn)}{d\hs^k}\right|^2|J_e|d\he
\leq C\sum_{|\alpha|=0}^{k-1}\|\nabla^\alpha\bq\|_e^2\leq Ch_e^{2-2k}\|\bq\|_e^2\leq Ch_e^{1-2k}\|\bq\|_T^2,
\end{equation}
where we have used the trace inequality (\ref{trace2}). Substituting (\ref{EQ:July05:003}) into (\ref{EQ:July05:001}) yields
\begin{equation}\nonumber
|\langle\phi-Q_b\phi,\bq\cdot\bn\rangle_e|\le Ch_e^{\frac12}\|\phi-Q_b\phi\|_\pT \ \|\bq\|_{T},
\end{equation}
which verifies the estimate (\ref{EQ:July06-2016:600}) for $k\ge1$. In the case of $k\ge2$, the inequality (\ref{EQ:July05:003}) can be replaced by
\begin{equation}\label{EQ:July05:003new}\nonumber
\int_{\he}\left|\frac{d^k({\widehat\bq}\cdot\bn)}{d\hs^k}\right|^2|J_e|d\he\leq Ch_e^{3-2k}(\|\bq\|_{T}^2+\|\nabla\bq\|_T^2),
\end{equation}
which, together with (\ref{EQ:July05:001}), verifies the second estimate in (\ref{EQ:July06-2016:600}) for $k\ge2$.

Finally, (\ref{EQ:July06-2016:800}) stems from (\ref{EQ:July06-2016:600}) with the following inequality
$$
\|v_0-Q_bv_0\|_\pT\leq Ch_e\|\nabla v_0\|_\pT\leq Ch_e^{\frac12}\|\nabla v_0\|_T,
$$
where we have used the trace inequality (\ref{trace2}). This completes the proof of the lemma.
\end{proof}

\begin{lemma} For any $v=\{v_0,v_b\}\in W_h$, there holds
\begin{eqnarray}
h_T^{-1}\|v_0-v_b\|_\pT^2&\leq&C(\|\nabla v_0\|_T^2+h_T^{-1}\|Q_bv_0-v_b\|_\pT^2),\label{happy.000}\\
\|\nabla v_0\|_T^2&\le&C\left(\|\nabla_w v\|_T^2+h_T^{-1}\|Q_bv_0-v_b\|_\pT^2\right),\label{happy.001}\\
\|\nabla_wv\|_T^2&\le&C\left(\|\nabla v_0\|_T^2+h_T^{-1}\|Q_bv_0-v_b\|_\pT^2\right),\label{happy.001s}
\end{eqnarray}
provided the meshsize $h$ is sufficiently small, where $C$ is a positive constant. Consequently, the discrete $H^1$-norm $\|\cdot\|_{1,h}$ is equivalent to the triple-bar norm $\3bar\cdot\3bar$ in the sense that there exist positive constants $\alpha_1$ and $\alpha_2$ such that
\begin{equation}\label{happy.001new}
\alpha_1\3barv\3bar\leq \|v\|_{1,h}\leq\alpha_2\3barv\3bar.
\end{equation}
\end{lemma}

\begin{proof} Note that on each edge $e\subset\pT$ one has
\begin{equation}\label{es1}
\|v_0-Q_bv_0\|_e^2=\int_\he|{\widehat v}_0-\widehat{Q}_b\widehat{v}_0|^2|J_e|d\he\leq Ch_e\|\nabla v_0\|_T^2.
\end{equation}
It follows from the triangle inequality that
\begin{eqnarray*}
\|v_0-v_b\|_\pT^2
&\leq&2\|v_0-Q_bv_0\|_\pT^2+2\|Q_bv_0-v_b\|_\pT^2\\
&\leq&Ch_T\|\nabla v_0\|_T^2+2\|Q_bv_0-v_b\|_\pT^2,
\end{eqnarray*}
where we used (\ref{es1}). This implies the inequality (\ref{happy.000}).

For any $v=\{v_0,v_b\}\in W_h$, it follows from (\ref{dwd-2}) that
\begin{equation}\label{March24-2013-01}
\begin{split}
(\nabla_wv,\bq)_T
&=(\nabla v_0,\bq)_T+\l v_b-v_0,\bq\cdot\bn\r_\pT\\
&=(\nabla v_0,\bq)_T+\l v_b-Q_bv_0,\bq\cdot\bn\r_\pT+\l Q_bv_0-v_0,\bq\cdot\bn\r_\pT,
\end{split}
\end{equation}
for all $\bq\in[P_{k-1}(T)]^2$. Thus, from the Cauchy-Schwarz inequality, the trace inequality (\ref{trace2}), and the estimate
(\ref{EQ:July06-2016:800}), we obtain from (\ref{March24-2013-01}) that
\begin{eqnarray*}
|(\nabla v_0,\bq)_T|
&\le&\|\nabla_wv\|_T\|\bq\|_T+\|Q_bv_0-v_b\|_\pT\|\bq\|_\pT+Ch_T\|\nabla v_0\|_T\|\bq\|_T\\
&\le&\|\nabla_wv\|_T\|\bq\|_T+Ch_T^{-\frac12}\|Q_bv_0-v_b\|_\pT\|\bq\|_T+Ch_T\|\nabla v_0\|_T\|\bq\|_T,
\end{eqnarray*}
which leads to
$$
\|\nabla v_0\|_T\leq C(\|\nabla_w v\|_T+h_T^{-\frac12}\|Q_bv_0-v_b\|_\pT +h_T\|\nabla v_0\|_T).
$$
This gives rise to the estimate (\ref{happy.001}) for sufficiently small $h_T$. The estimate (\ref{happy.001s}) can be derived in a similar, but simpler fashion. (\ref{happy.001new}) can be obtained easily using (\ref{happy.000})-(\ref{happy.001s}).
\end{proof}

\medskip

The following two Lemmas contain some useful estimates for the local $L^2$ projection operators.

\begin{lemma}
For any $w\in H^{m+1}(T)$, $m\in[0,k]$, there holds
\begin{equation}\label{k2}
\|\nabla w-\bbQ_h\nabla w\|_{T}+\|\nabla(w-Q_0w)\|_{T}+h_T^{-1}\|w-Q_0 w\|_T\le Ch_T^{m}\|w\|_{m+1,T}.
\end{equation}
\end{lemma}

\begin{lemma} There holds
\begin{equation}\label{k3}
\|Q_0w-w\|_\pT\le Ch_T^{m+\frac12}\|w\|_{m+1,T},\forall w\in H^{m+1}(T),m\in[0,k],
\end{equation}
\begin{equation}
\|Q_bw-w\|_\pT\le Ch_T^{m-\frac12}\|w\|_{m,T},\forall w\in H^{m}(T),m\in[1,k].\label{k4}
\end{equation}
\end{lemma}

\begin{proof}
From the trace inequality (\ref{trace}) and the estimate (\ref{k2}) we have
\begin{eqnarray*}
\|Q_0w-w\|_\pT
&\le&C(h_T^{-1}\|Q_0w-w\|_T^2+h_T\|\nabla(Q_0w-w)\|_T^2)^{1/2}\\
&\le&Ch_T^{m+\frac12}\|w\|_{m+1,T},
\end{eqnarray*}
which verifies (\ref{k3}).

Next, for any edge $e\subset\pT$, using the mapping $\bx=F_e(\hat{s})$ we arrive at
\begin{eqnarray*}
\|Q_bw-w\|_e^2=\int_e(Q_bw-w)^2de=\int_\he|(Q_b w-w)\circ F_e|^2|J_e|d{\hs},
\end{eqnarray*}
where $J_e$ is the Jacobian of the mapping. Note that
$$
\hw=w\circ F_e,\quad\hQ_b\hw=(Q_bw)\circ F_e,
$$
where $\hQ_b$ is the $|J_e|$-weighted $L^2(\he)$ projection onto the polynomial space of degree $\ell_k=k-1$ on $\he$. Thus, for $m\in[1,k]$ we have
\begin{equation}\label{hello.there}
\|Q_bw-w\|_e^2=\int_\he(\hQ_b\hw-\hw)^2|J_e|d{\hs}\leq Ch^{2m-1}\|\hw\|_{m,\hat T}^2,
\end{equation}
where $\hat T$ is the image of the pyramid $\he$ as its base. By mapping back to the element $T$, we obtain the desired estimate (\ref{k4}).
\end{proof}

\medskip

\begin{lemma} Assume that the curved finite element partition $\T_h$ is shape regular. For any $w\in H^{k+1}(\Omega)$ and $v=\{v_0,v_b\}\in W_h$, we have
\begin{eqnarray}
|s(Q_hw,v)|&\le&Ch^k\|w\|_{k+1}\3bar v\3bar,\label{mmm1}\\
\left|\ell_1(w,v)\right|&\leq&Ch^k\|w\|_{k+1}\3barv\3bar,\label{mmm2}\\
\left|\ell_2(w,v)\right|&\leq&Ch^k\|w\|_{k}\3barv\3bar.\label{mmm3}
\end{eqnarray}
\end{lemma}

\begin{proof} To derive (\ref{mmm1}), we use the definition of $s(\cdot,\cdot)$, the Cauchy-Schwarz inequality, and the estimate (\ref{k3}) with $m=k$ to obtain
\begin{eqnarray*}
|s(Q_hw, v)|
&=&\left|\sum_{T\in\T_h}h_T^{-1}\langle Q_b(Q_0w)-Q_bw,\;Q_bv_0-v_b\rangle_\pT\right|\\
&=&\left|\sum_{T\in\T_h}h_T^{-1}\langle Q_b(Q_0w-w),\;Q_bv_0-v_b\rangle_\pT\right|\\
&\le&\left(\sum_{T\in\T_h}h_T^{-1}\|Q_0w-w\|^2_{\pT}\right)^{\frac12}
      \left(\sum_{T\in\T_h}h_T^{-1}\|Q_bv_0-v_b\|^2_{\pT}\right)^{\frac12}\\
&\le&Ch^k\|w\|_{k+1}\3barv\3bar.
\end{eqnarray*}

As to (\ref{mmm2}), from the Cauchy-Schwarz inequality, the trace inequality (\ref{trace}), and the estimate (\ref{k2}) we have
\begin{equation}\label{EQ:July7-2016:001}
\begin{split}
|\ell_1(w,v)|
=&\left|\sum_{T\in\T_h}\langle(\nabla w-\bbQ_h\nabla w)\cdot\bn,v_0-v_b\rangle_\pT\right|\\
\le&\left(\sum_{T\in\T_h} h_T\|\nabla w-\bbQ_h\nabla w\|_{\pT}^2\right)^{\frac12}
    \left(\sum_{T\in\T_h} h_T^{-1}\|v_0-v_b\|_\pT^2\right)^{\frac12}\\
\le&Ch^k\|w\|_{k+1}\left(\sum_{T\in\T_h}h_T^{-1}\|v_0-v_b\|_\pT^2\right)^{\frac12}.
\end{split}
\end{equation}
Note that (\ref{happy.000}) and (\ref{happy.001}) implies
$$
\sum_{T\in\T_h}h_T^{-1}\|v_0-v_b\|_\pT^2\leq C\3barv\3bar^2.
$$
Substituting the above into (\ref{EQ:July7-2016:001}) yields the estimate (\ref{mmm2}).

To establish (\ref{mmm3}), we use the estimate (\ref{EQ:July06-2016:600}) and Cauchy-Schwarz inequality to obtain
\begin{eqnarray*}
|\ell_2(w,v)|
&\leq&\sum_{T\in\T_h}\left|\l Q_bw-w,\nabla_wv\cdot\bn\r_\pT\right|\\
&\le&C\sum_{T\in\T_h}h_T^{\frac12}\|Q_bw-w\|_\pT\|\nabla_wv\|_T \\
&\le&Ch^{\frac12}\left(\sum_{T\in\T_h}\|Q_bw-w\|_\pT^2\right)^{\frac12}\left(\sum_{T\in\T_h}\|\nabla_wv\|_T^2\right)^{\frac12}.
\end{eqnarray*}
Now using (\ref{k4}) with $m=k$ we have the following estimate
\begin{eqnarray*}
|\ell_2(w,v)|\leq Ch^k\|w\|_k\3barv\3bar.
\end{eqnarray*}
This completes the proof of the lemma.
\end{proof}

\section{Error Estimates}\label{Section:H1ErrorEstimate}

With the help of the error equation (\ref{ee}) and the technical estimates presented in the previous section, we are ready to present some optimal order error estimates for the weak Galerkin finite element solution in discrete $H^1$-norm and $L^2$-norm.

\begin{theorem}\label{H1-errorestimate}
Let $u_h\in W_h$ be the weak Galerkin finite element solution of the problem (\ref{pde})-(\ref{bc}) arising from (\ref{wg}). Assume the exact solution $u\in H^{k+1}(\Omega)$ and meshsize $h$ is sufficiently small. Then, there exists a constant $C$ such that
\begin{equation}\label{err1}
\3baru_h-Q_hu\3bar\le Ch^k\|u\|_{k+1}.
\end{equation}
\end{theorem}

\begin{proof}
By letting $v=e_h$ in (\ref{ee}), we have
\begin{eqnarray*}
\3bare_h\3bar^2=\ell_1(u,e_h)+\ell_2(u,e_h)+s(Q_hu,e_h).
\end{eqnarray*}
It then follows from the estimates (\ref{mmm1})-(\ref{mmm3}) that\[\3bare_h\3bar^2\le Ch^k\|u\|_{k+1}\3bare_h\3bar,\] which implies
(\ref{err1}). This completes the proof of the theorem.
\end{proof}

\medskip

Using the norm equivalence (\ref{happy.001new}) and the error estimate (\ref{err1}) we immediately obtain
\begin{equation}\label{err1-H1}
\|u_h-Q_hu\|_{1,h}\le Ch^k\|u\|_{k+1}.
\end{equation}
Furthermore, with the straightforward extension of $\|\cdot\|_{1,h}$ to general weak functions we arrive at
\begin{equation}\label{err1-H1-new}
\|u_h-u\|_{1,h}\le Ch^k\|u\|_{k+1}.
\end{equation}


Now we turn to deriving an optimal order error estimate for the weak Galerkin finite element approximation in the $L^2$ norm by following the usual duality argument. To this end, consider the dual problem which seeks $\Phi\in H_0^1(\Omega)$ satisfying
\begin{eqnarray}
-\Delta\Phi=e_0\quad \mbox{in}\;\Omega.\label{dual}
\end{eqnarray}
Recall that $e_0=Q_0u-u_0$ is the first component of the error function $e_h$. Assume that the dual problem (\ref{dual}) has $H^{2}$-regularity in the sense that there exists a constant $C$ such that
\begin{equation}\label{reg}
\|\Phi\|_2\le C\|e_0\|_0.
\end{equation}

Throughout the following estimates, we assume that all the interior edges of the curved finite element partition $\T_h$ are straight line segments. In other words, the curved edges only appear on the boundary of the domain. This assumption is practically feasible and computationally preferable. In addition, we shall consider only the finite element solution of order $k\ge2$, as no need is necessary for curved elements of lowest order $k=1$.

\begin{theorem}\label{L2-errorestimate}
Let $u_h\in W_h$ be the weak Galerkin finite element solution of the problem (\ref{pde})-(\ref{bc}) arising from (\ref{wg}) with order $k\ge2$. Assume that the exact solution of (\ref{pde})-(\ref{bc}) is sufficiently regular such that $u\in H^{k+1}(\Omega)$ and the meshsize $h$ is sufficiently small. Then there exists a constant $C$ such that
\begin{equation}\label{err2}
\|u-u_h\|\le Ch^{k+1}\|u\|_{k+1}.
\end{equation}
\end{theorem}

\begin{proof}
By testing (\ref{dual}) against $e_0$ we obtain
\begin{eqnarray}\nonumber
\|e_0\|^2
&=&-(\Delta\Phi,e_0)\\
&=&\sum_{T\in\T_h}(\nabla\Phi,\ \nabla e_0)_T-\sum_{T\in\T_h}\l\nabla\Phi\cdot\bn,\ e_0-e_b\r_{\pT},\label{jw.08}
\end{eqnarray}
where we have used the fact that $e_b=0$ on $\partial\Omega$. Next, by setting $w=\Phi$ and $v=e_h$ in (\ref{April-5:888}) we arrive at
\begin{equation}\label{April-5:200}
\begin{split}
\sum_{T\in\T_h}(\nabla_wQ_h\Phi,\nabla_we_h)_T
=&\sum_{T\in\T_h}(\nabla\Phi,\nabla e_0)_T+\sum_{T\in\T_h}\l\nabla\Phi\cdot\bn,e_b-e_0\r_\pT \\
 &+\ell_1(\Phi,e_h)+\ell_2(\Phi,e_h).
\end{split}
\end{equation}
Substituting (\ref{April-5:200}) into (\ref{jw.08}) gives
\begin{eqnarray}
\|e_0\|^2=(\nabla_we_h,\ \nabla_wQ_h\Phi)-\ell_1(\Phi,e_h)-\ell_2(\Phi,e_h).\label{m2}
\end{eqnarray}
Now using the error equation (\ref{ee}) we have
\begin{eqnarray}
(\nabla_we_h,\ \nabla_w Q_h\Phi)
&=&\ell_1(u,Q_h\Phi)+\ell_2(u,Q_h\Phi)\nonumber\\
&&+s(Q_hu,\ Q_h\Phi)-s(e_h,\ Q_h\Phi).\label{m3}
\end{eqnarray}
Combining (\ref{m2}) with (\ref{m3}) yields
\begin{equation}\label{m4-new}
\begin{split}
\|e_0\|^2
=&\ell_1(u,Q_h\Phi)+\ell_2(u,Q_h\Phi)+s(Q_hu,\ Q_h\Phi)\\
 &-s(e_h,\ Q_h\Phi)-\ell_1(\Phi,e_h)-\ell_2(\Phi,e_h)\\
=&\sum_{j=1}^6I_j,
\end{split}
\end{equation}
where $I_j$ are defined accordingly. The rest of the proof shall deal with the terms $I_j (j=1,\cdots,6)$ one by one.

\medskip
\noindent{\bf Step 1:} \ Note that all the interior edges are straight line segments on which $Q_b$ is the usual $L^2$ projection onto $P_{k-1}(e)$. Thus, we have
$$
\langle(\nabla u-\bbQ_h\nabla u)\cdot\bn,\;\Phi-Q_b\Phi\rangle_{\pT\cap\E_h^0}=\langle\nabla u\cdot\bn,\;\Phi-Q_b\Phi\rangle_{\pT\cap\E_h^0},
$$
which, together with the fact that $\Phi=0$ and $Q_b\Phi=0$ on the boundary $\partial\Omega$, leads to
$$
\sum_{T\in\T_h}\langle(\nabla u-\bbQ_h\nabla u)\cdot\bn,\;\Phi-Q_b\Phi\rangle_{\pT}
=\sum_{T\in\T_h}\langle\nabla u\cdot\bn,\;\Phi-Q_b\Phi\rangle_{\pT}=0.
$$
It follows from Cauchy-Schwarz inequality that
\begin{equation}\label{EQ:step1}
\begin{split}
|I_1|=&|\ell_1(u,Q_h\Phi)|\\
=&\left|\sum_{T\in\T_h}\langle(\nabla u-\bbQ_h\nabla u)\cdot\bn,\;Q_0\Phi-Q_b\Phi\rangle_\pT\right|\\
=&\left|\sum_{T\in\T_h}\langle(\nabla u-\bbQ_h\nabla u)\cdot\bn,\; Q_0\Phi-\Phi\rangle_\pT \right|\\
\leq&\Big(\sum_{T\in\T_h}\|\nabla u-\bbQ_h\nabla u\|^2_\pT\Big)^{\frac{1}{2}}\Big(\sum_{T\in\T_h}\|Q_0\Phi-\Phi\|^2_\pT\Big)^{\frac{1}{2}}\\
\leq&Ch^{k+1}\|u\|_{k+1}\|\Phi\|_2,
\end{split}
\end{equation}
where we have used the trace inequality (\ref{trace}) and the estimates (\ref{k2}) in the last line.

\medskip
\noindent{\bf Step 2:} \ To bound $I_2=\ell_2(u,Q_h\Phi)$, we use the second estimate in (\ref{EQ:July06-2016:600}) to obtain
\begin{equation}\label{EQ:step2-pre}
\begin{split}
 &|\ell_2(u,Q_h\Phi)|\\
=&\left|\sum_{T\in\T_h}\l u-Q_bu,(\nabla_wQ_h\Phi)\cdot\bn\r_\pT\right|\\
\le&Ch^{3/2}\sum_{T\in\T_h}\|u-Q_bu\|_\pT(\|\nabla_wQ_h\Phi\|_{1,T}+\|\nabla_wQ_h\Phi\|_{T})\\
\le&Ch^{k+1}\|u\|_{k}\big\{\left(\sum_{T\in\T_h}\|\nabla_wQ_h\Phi\|_{1,T}^2\right)^{\frac12}
      +\left(\sum_{T\in\T_h}\|\nabla_wQ_h\Phi\|_{T}^2\right)^{\frac12}\big\},
\end{split}
\end{equation}
where we have employed Cauchy-Schwarz inequality and the estimate (\ref{k4}) with $m=k$ in the last inequality. By assumption, all the interior edges are straight line segments. Using this and the fact that $\Phi|_{\partial\Omega}=0$ and $Q_b\Phi|_{\partial\Omega}=0$ we can see that the boundary integral on the right-hand side of (\ref{key}) vanishes, and hence $\nabla_wQ_h\Phi=\bbQ_h\nabla\Phi$. It follows that
$$
\left(\sum_{T\in\T_h}\|\nabla_w Q_h\Phi\|_{1,T}^2\right)^{\frac12}
=\left(\sum_{T\in\T_h}\|\bbQ_h\nabla\Phi\|_{1,T}^2\right)^{\frac12}\leq C\|\Phi\|_2.
$$
Similarly,
$$
\left(\sum_{T\in\T_h}\|\nabla_w Q_h\Phi\|_{T}^2\right)^{\frac12}=\left(\sum_{T\in\T_h}\|\bbQ_h\nabla\Phi\|_{T}^2\right)^{\frac12}\leq C\|\Phi\|_2.
$$
Substituting the above two estimates into (\ref{EQ:step2-pre}) yields
\begin{equation}\label{EQ:step2}
|\ell_2(u,Q_h\Phi)|\leq Ch^{k+1}\|u\|_{k}\|\Phi\|_2.
\end{equation}

\medskip
\noindent{\bf Step 3:} \ As to the third term $I_3$, we use Cauchy-Schwarz inequality, the $L^2$-boundedness of $Q_b$ and the estimate (\ref{k3}) to obtain
\begin{equation}\label{EQ:step3}
\begin{split}
\left|s(Q_hu,\; Q_h\Phi)\right|
\le&\sum_{T\in\T_h}h_T^{-1}\left|\l Q_bQ_0u-Q_bu,\ Q_bQ_0\Phi-Q_b\Phi\r_\pT\right|\\
\le&\left(\sum_{T\in\T_h}h_T^{-2}\|Q_0u-u\|^2_\pT\right)^{\frac12}\left(\sum_{T\in\T_h}\|Q_0\Phi-\Phi\|^2_\pT\right)^{\frac12} \\
\le&Ch^{k+1}\|u\|_{k+1}\|\Phi\|_2.
\end{split}
\end{equation}

\medskip
\noindent{\bf Step 4:} \ For the term $|I_4|=|s(e_h,\ Q_h\Phi)|$, from the estimate (\ref{mmm1}) (with $k=1$ and $w=\Phi$) we obtain
\begin{eqnarray}\label{EQ:step4}
|s(e_h,\ Q_h\Phi)|\le Ch\3bare_h\3bar\|\Phi\|_2.
\end{eqnarray}
As to the term $|I_5|=|\ell_1(\Phi,e_h)|$, it follows from (\ref{mmm2}) with $w=\Phi$ and $k=1$ that
\begin{eqnarray}\label{EQ:step5}
|\ell_1(\Phi,e_h)|\le Ch\3bare_h\3bar\|\Phi\|_2.
\end{eqnarray}
Finally, for the term $|I_6|=|\ell_2(\Phi,e_h)|$, we use the estimate (\ref{mmm3}) with $w=\Phi$ and $k=1$ to get
\begin{eqnarray}\label{EQ:step6}
|\ell_2(\Phi,e_h)|\le Ch\3bare_h\3bar\|\Phi\|_2.
\end{eqnarray}

Substituting the estimates (\ref{EQ:step1}) and (\ref{EQ:step2})-(\ref{EQ:step6}) into (\ref{m4-new}) yields
$$
\|e_0\|^2\leq C(h^{k+1}\|u\|_{k+1}+h\3bar e_h\3bar)\|\Phi\|_2,
$$
which, combined with the regularity assumption (\ref{reg}) and the error estimate (\ref{err1}), gives the optimal order error estimate
(\ref{err2}).
\end{proof}

\medskip

For any weak finite element function $v=\{v_0,v_b\}\in W_h$, we define the following semi-norm
\begin{equation*}\label{ebnorm}
\|v_b\|_{\E_h}=\left(\sum_{T\in{\cal T}_h}h_T\|v_b\|^2_{\partial T}\right)^{\frac{1}{2}}.
\end{equation*}
By combining the $L^2$ error estimate (\ref{err2}) with the $H^1$ error estimate (\ref{err1}), one can derive the following $L^2$ error estimate for the WG approximation on the boundary of each element. Details of the proof are left to interested readers as an exercise.

\begin{theorem}
Let $u_h\in W_h$ be the solution of the weak Galerkin Algorithm (\ref{wg}) with finite elements of order $k\geq2$. Assume that the exact solution $u$ of (\ref{pde})-(\ref{bc}) is sufficiently regular such that $u\in H^{k+1}(\Omega)$. Assume that curved edges in the finite element partition can only appear on the boundary of the domain. There exists a constant $C$ such that
\begin{equation*}\label{err3}
\|Q_bu-u_b\|_{\E_h}\leq Ch^{k+1}\|u\|_{k+1}.
\end{equation*}
\end{theorem}

\section{Numerical Integration  on a Curved Polygon}\label{Section:NC} 
Let $T$ be a curved element. For simplicity of implementation, assume that the boundary $\pa T$ of the curved element $T$ consists of one curved edge $e_1$ and  the rest edges $e_i (i=2,\ldots,N_E)$ being straight edges. For any given function $F(x, y)$ defined on the curved element $T$, it follows from the Taylor expansion that
\begin{equation*}
\begin{split}
F(x,y)=&F(x_T,y_T)+\pa_xF(x_T,y_T)\cdot(x-x_T)+\pa_yF(x_T,y_T)\cdot(y-y_T)\\
       &+\cdots+\frac{\Big((x-x_T)\pa_x+(y-y_T)\pa_y\Big)^kF(x_T,y_T)}{k!}+O(h^{k+1}),
\end{split}
\end{equation*}
where $(x_T,y_T)$ is a given point on the curved element $T$. Thus, one arrives at
\begin{equation}\label{NC:Jan-1}
\begin{split}
&\int_TF(x,y)dT
= F(x_T,y_T)\int_T1dT\\&+\pa_xF(x_T,y_T)\int_T(x-x_T)dT+\pa_yF(x_T,y_T)\int_T(y-y_T)dT\\
 &+\cdots+\frac{\int_T\Big((x-x_T)\pa_x+(y-y_T)\pa_y\Big)^kF(x_T,y_T)dT}{k!}+{\cal O}(h^{k+1}).
\end{split}
\end{equation}
The number of terms you shall expand for the right hand side of  \eqref{NC:Jan-1} depends on the desired approximation accuracy.

For simplicity, we shall consider the first three terms on the right hand of \eqref{NC:Jan-1} to calculate the approximation of the integral $\int_TF(x,y)dT$ with the  approximation error ${\cal O}(h^2)$.  
As to the first integral $\int_T1dT$, there exists a vector-valued function ${\bf f_1}(x,y)$ such that $\nabla\cdot{\bf f_1}(x,y)=1$. Then, it follows from the divergence theorem that
\begin{equation}\label{NC:Jan-2}
\begin{split}
\int_T1dT=&\int_{T}\nabla\cdot{\bf f_1}(x,y)dT 
         = \int_{\pa T}{\bf f_1}(x,y)\cdot\bn ds\\
         =&\int_{e_1}{\bf f_1}(x,y)\cdot\bn_1ds+\sum_{i=2}^{N_E}\int_{e_i}{\bf f_1}(x,y)\cdot\bn_ids,
\end{split}
\end{equation}
where $\bn_i$ represents the unit outward normal direction to edge $e_i$ for $i=1,\ldots,N_E$.

Similarly, there exist two functions ${\bf f_2}(x,y)$ and ${\bf f_3}(x,y)$ such that $\nabla\cdot{\bf f_2}(x,y)=x-x_T$ and $\nabla\cdot{\bf f_3}(x,y)=y-y_T$. Thus, we arrive at
\begin{equation}\label{NC:Jan-3}
\begin{split}
\int_T(x-x_T)dT=&\int_{e_1}{\bf f_2}(x,y)\cdot\bn_1 ds+\sum_{i=2}^{N_E}\int_{e_i}{\bf f_2}(x,y)\cdot\bn_i ds,
\end{split}
\end{equation}
\begin{equation}\label{NC:Jan-4}
\int_T(y-y_T)dT=\int_{e_1}{\bf f_3}(x,y)\cdot\bn_1 ds+\sum_{i=2}^{N_E}\int_{e_i}{\bf f_3}(x,y)\cdot\bn_i ds.
\end{equation}

Recall that the parametric presentation for $e_i$ is given by $(x,y)=(\phi(\hat{s}),\psi(\hat{s}))$. For any given point $(x,y)\in e_i$, there exists a vector $\pmb{\beta}$ starting from the given point to a point in the interior of the element $T$. Then one arrives at the unit outward normal direction to $e_i$ given by   $$\bn_i=\frac{\alpha(\frac{d(\psi(d\hs))}{d\hs};-\frac{d(\phi(d\hs))}{d\hs})}{\sqrt{(\frac{d(\psi(d\hs))}{d\hs})^2+(\frac{d(\phi(d\hs))}{d\hs})^2   }}$$   where the coefficient $\alpha$ is set by
 \begin{equation*}
\begin{split}
& \alpha=\begin{cases}1,& \mbox{if}~~(\frac{d(\psi(d\hs))}{d\hs};-\frac{d(\phi(d\hs))}{d\hs})\cdot\pmb{\beta}<0,\\
  -1,& \mbox{otherwise.}\end{cases}
\end{split}
\end{equation*} 

Recall that the mapping $\hat{F_e}$ maps a curved edge $e$ to a straight edge $\hat{e}$. Then, for any $v_b$, $w_b\in V_b(e,k-1)$, we have
\begin{equation}\label{curveintegral}
\begin{split}
\int_ev_bw_bds=&\int_e\hat{v}_b\hat{w}_b\circ\hat{F_e}de\\
              =&\int_{\hat{e}}\hat{v}_b\hat{w}_b\cdot |J_e|d{\hs}\\
              =&\int_{\hat{e}}\hat{v}_b\hat{w}_b\sqrt{(\phi'(\hat{s}))^2+(\psi'(\hat{s}))^2}d\hs,
\end{split}
\end{equation}
which can be computed by using numerical integration with desired  precision.

Substituting \eqref{NC:Jan-2}-\eqref{NC:Jan-4} into \eqref{NC:Jan-1} gives rise to an approximation for the integral $\int_TF(x,y)dT$ with the  approximation error ${\cal O}(h^2)$, which will be further calculated by \eqref{curveintegral} and numerical integration with required precision.

\section{Numerical Experiments}\label{Section:NE}
This section shall illustrate several numerical experiments to demonstrate the accuracy and efficiency of the curved elements in WG methods. For simplicity of implementation, we consider two types of WG element. One is called a curved WG element where one edge is curved and the rest edges are straight; the other is called a straight WG element where all edges are straight. 
The curved WG element with degree $k$ and the discrete weak gradient discretized by $[P_{k-1}(T)]^2$ is denoted by $P_k(T)-V_b(\pa T,k-1)-[P_{k-1}(T)]^2$ element. Analogously, the straight WG element with degree $k$ is denoted by $P_k(T)-P_{k-1}(\pa T)-[P_{k-1}(T)]^2$ element. 

{\bf Test case 1 (curved quadrilateral domain)} We consider a  curved quadrilateral domain given by 
$$\O=\{\{x,y\}:~~0\leq x\leq1,~~g_1(x)\leq y\leq g_2(x)\},$$
where $g_1(x)=\frac{1}{20}\sin(\pi x)$ and $g_2(x)=1+\frac{1}{20}\sin(3\pi x)$. The exact solution is $u=x(x-1)(y-g_1(x))(y-g_2(x))$ shown in Figure \ref{Pictu-example1} (a). The finite element partition on the curved domain $\O$ is constructed such that the mesh node  $(x_{\O},y_{\O})$ is  given by \cite{LM2021} 
 \begin{equation*}
\begin{split}
& (x_{\O},y_{\O})=\begin{cases}(x_s,y_s+g_1(x_s)(1-2y_s)),& \mbox{if}~~y_s\leq \frac{1}{2},\\
  (x_s,1-y_s+g_2(x_s)(2y_s-1),& \mbox{otherwise,}\end{cases}
\end{split}
\end{equation*} 
where $(x_s,y_s)$ is the mesh point obtained by uniformly dividing the unit square domain $[0,1]^2$ into $n\times n$ sub-squares. The numerical tests are implemented on the curved uniform meshes and the straight uniform meshes respectively. The curved uniform meshes and straight uniform meshes are obtained by connecting the mesh nodes on the  curved boundary edges $y=g_1(x)$ and $y=g_2(x)$ where $0\leq x \leq 1$ by curved segments and straight segments, respectively, while the interior mesh nodes are both connected by straight edges. The first level of straight uniform meshes is shown in Figure \ref{Pictu-example1-1} (Left). The second level of the straight uniform meshes is refined by connecting the midpoints of the quadrilateral elements on the first level ending up with  dividing each quadrilateral element on the first level into $4$ sub quadrilateral elements as shown in Figure \ref{Pictu-example1-1} (Right). Similarly, the first two levels of curved uniform meshes are shown in Figure \ref{Pictu-example1-8}.

\begin{figure}[!htbp]
\centering
\includegraphics[height=0.36\textwidth,width=0.36\textwidth]{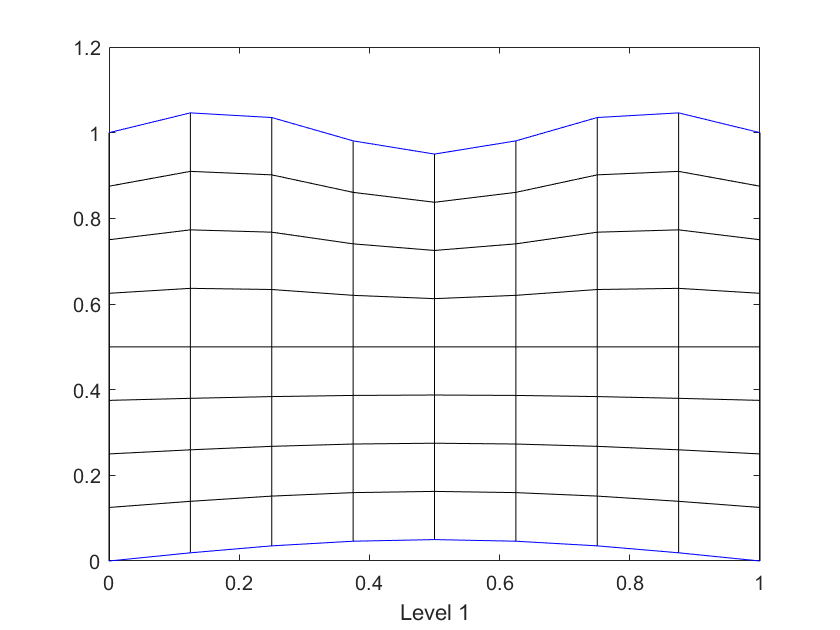}
\includegraphics[height=0.36\textwidth,width=0.36\textwidth]{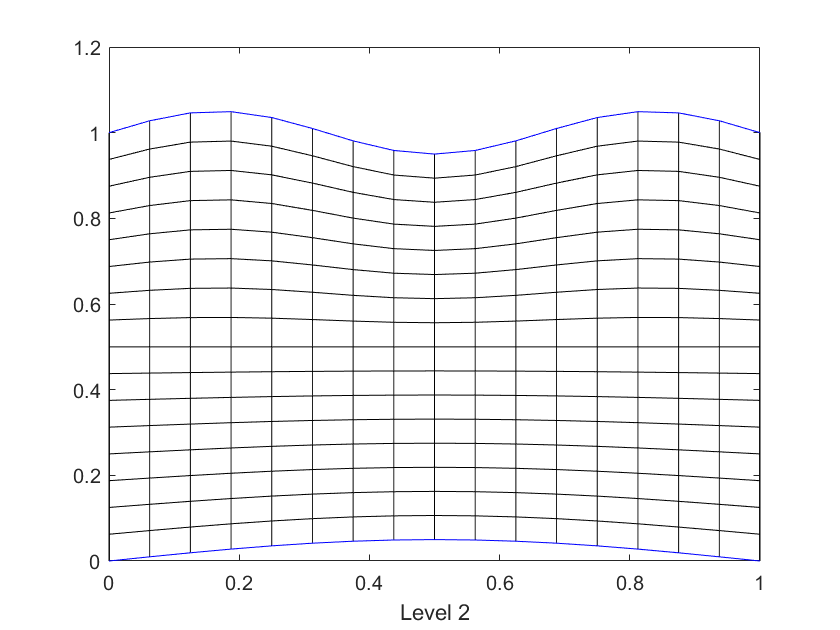}
\caption{Level 1 (Left) and level 2 (Right) of straight uniform meshes in test case 1.}\label{Pictu-example1-1}
\end{figure}

\begin{figure}[!htbp]
\centering
\includegraphics[height=0.36\textwidth,width=0.36\textwidth]{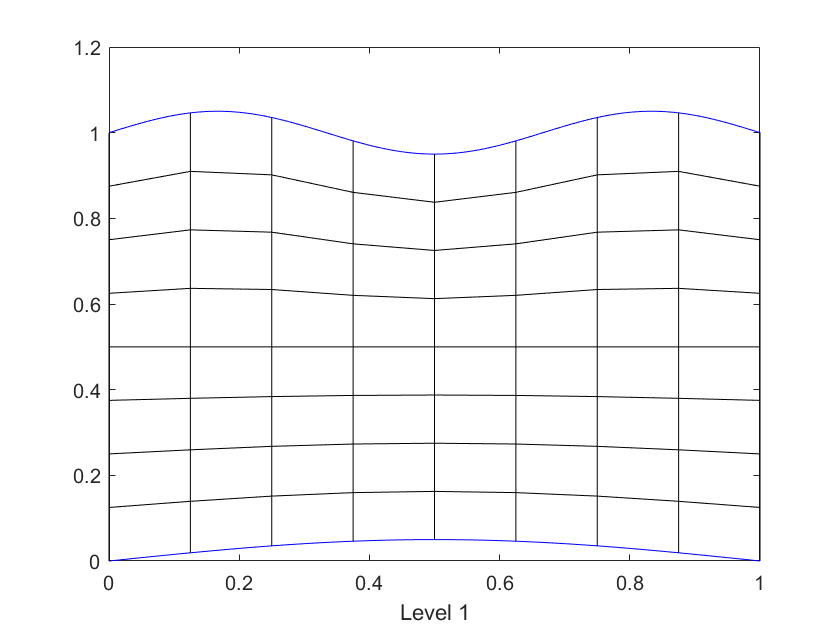}
\includegraphics[height=0.36\textwidth,width=0.36\textwidth]{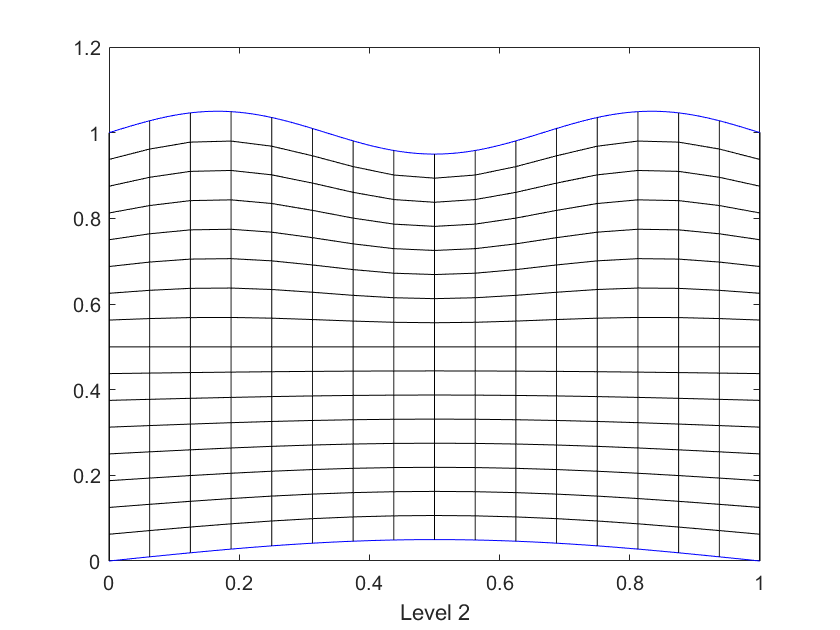}
\caption{Level 1 (Left) and level 2 (Right) of curved uniform meshes in test case 1.}\label{Pictu-example1-8}
\end{figure}

\begin{figure}[!htbp]
\centering
\includegraphics[height=0.40\textwidth,width=0.42\textwidth]{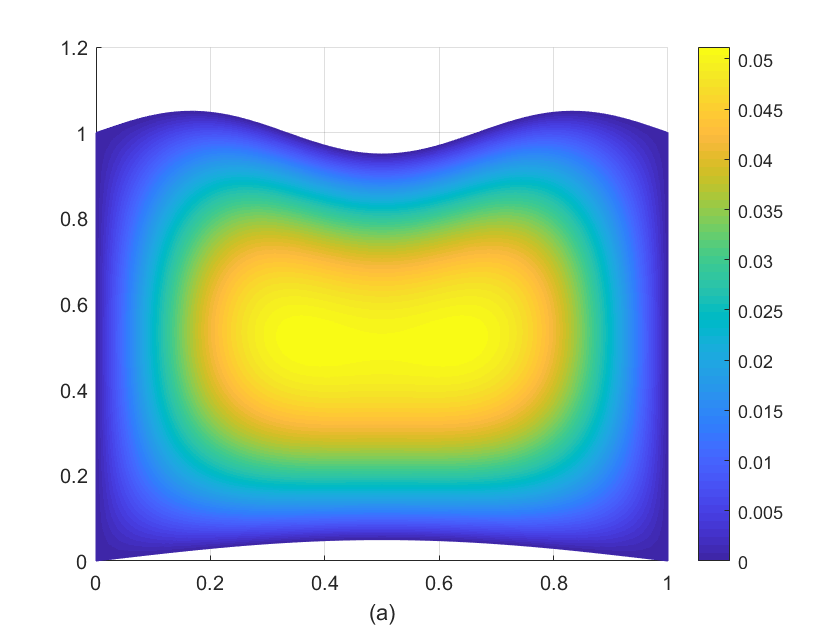}
\includegraphics[height=0.40\textwidth,width=0.42\textwidth]{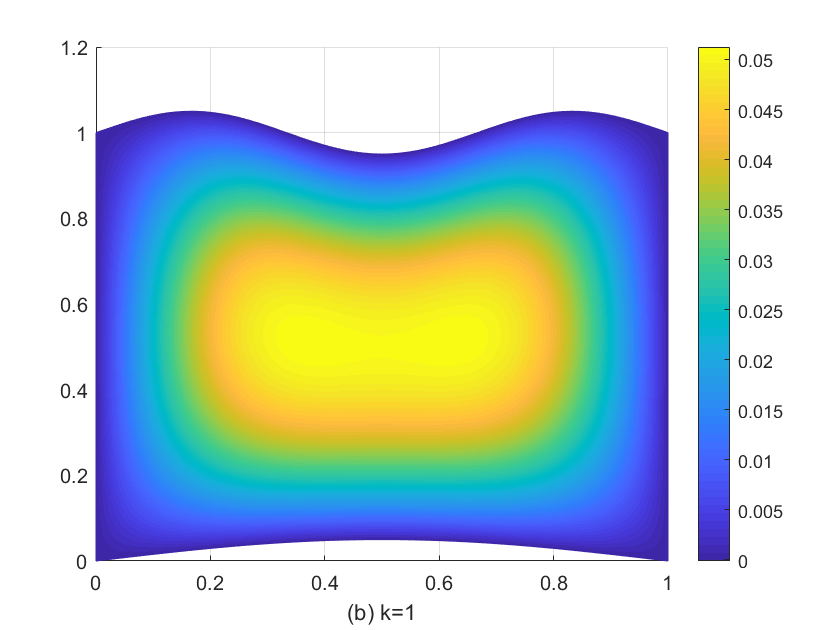}
\includegraphics[height=0.40\textwidth,width=0.42\textwidth]{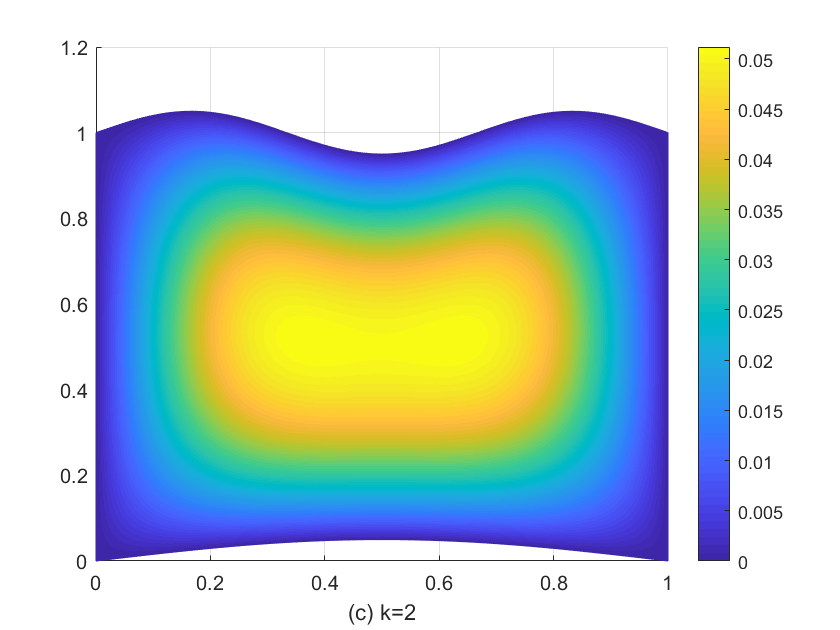}
\includegraphics[height=0.40\textwidth,width=0.42\textwidth]{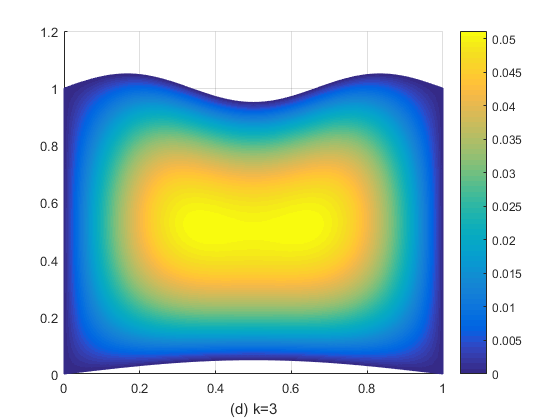}
\caption{(a) the exact solution $u$, (b) WG solution for  $k=1$, (c) WG solution for   $k=2$, (d) WG solution for  $k=3$.}\label{Pictu-example1}
\end{figure}

 We have observed from Table \ref{NE:exampe-1} that the optimal order of convergence for the numerical approximation on both the curved uniform meshes and the straight uniform meshes when the lowest order WG element $k=1$ is employed, which is consist with what the theory predicts; in addition, the convergence order of the WG numerical approximation in various norms on the curved uniform meshes is better than that on the straight uniform meshes for the higher order WG elements $k=2$ and $k=3$ respectively. The WG numerical solution $u_0$ on the curved uniform meshes for $k=1, 2, 3$ are illustrated in Figure \ref{Pictu-example1}.

\begin{table}[h]
\caption{Test case 1: Numerical errors and corresponding convergence rates.}\label{NE:exampe-1}
\begin{tabular}{|p{0.4cm}p{1.28cm}p{0.9cm}p{1.28cm}p{0.9cm}p{1.28cm}p{0.9cm}p{1.28cm}p{0.6cm}|}
\hline
1/h&$\3bare_h\3bar$&Rate&$\|e_0\|$&Rate&$\|e_b\|_{\E_h}$&Rate&$\|\nabla e_0\|$&Rate\\
\hline
&\text{$P_1(T)-V_b(\pa T,0)-[P_0(T)]^2$~element~on~the~curved~uniform~meshes}&&&&&&&\\
\hline
8           &5.08e-02  &-       &3.06e-03  &-       &3.25e-03  &-       &1.85e-02  &-\\
16          &2.82e-02  &0.85    &7.94e-04  &1.94    &9.15e-04  &1.83    &5.86e-03  &1.66\\
32          &1.47e-02  &0.94    &2.02e-04  &1.98    &2.38e-04  &1.94    &2.10e-03  &1.48\\
64          &7.52e-03  &0.97    &5.07e-05  &1.99    &6.02e-05  &1.98    &8.55e-04  &1.30\\
128         &3.79e-03  &0.99    &1.27e-05  &2.00    &1.51e-05  &2.00    &3.85e-04  &1.15 \\
\hline
&\text{$P_1(T)-P_0(\pa T)-[P_0(T)]^2$~element~on~the~straight~uniform~meshes}&&&&&&&\\
\hline
8           &5.08e-02  &-       &3.04e-03  &-       &3.22e-03  &-       &1.87e-02  &-\\
16          &2.85e-02  &0.84    &8.02e-04  &1.92    &9.27e-04  &1.79    &5.99e-03  &1.64\\
32          &1.50e-02  &0.93    &2.05e-04  &1.97    &2.43e-04  &1.93    &2.14e-03  &1.49\\
64          &7.64e-03  &0.97    &5.15e-05  &1.99    &6.17e-05  &1.98    &8.62e-04  &1.31\\
128         &3.86e-03  &0.99    &1.29e-05  &2.00    &1.55e-05  &1.99    &3.87e-04  &1.16 \\
\hline
&\text{$P_2(T)-V_b(\pa T,1)-[P_1(T)]^2$~element~on~the~curved~uniform~meshes}&&&&&&&\\
\hline
8           &1.24e-02  &-       &3.43e-04  &-       &1.01e-03  &-       &9.41e-03  &-\\
16          &3.24e-03  &1.94    &4.16e-05  &3.04    &1.43e-04  &2.82    &2.30e-03  &2.04\\
32          &8.37e-04  &1.96    &5.19e-06  &3.00    &1.90e-05  &2.92    &5.75e-04  &2.00\\
64          &2.12e-04  &1.98    &6.51e-07  &3.00    &2.44e-06  &2.96    &1.44e-04  &2.00\\
128         &5.35e-05  &1.99    &8.15e-08  &3.00    &3.08e-07  &2.98    &3.61e-05  &2.00 \\
\hline
&\text{$P_2(T)-P_1(\pa T)-[P_1(T)]^2$~element~on~the~straight~uniform~meshes}&&&&&&&\\
\hline
8           &1.22e-02  &-       &3.73e-04  &-       &1.04e-03  &-       &8.88e-03  &-\\
16          &3.42e-03  &1.83    &6.25e-05  &2.58    &1.66e-04  &2.65    &2.16e-03  &2.04\\
32          &1.10e-03  &1.64    &1.32e-05  &2.24    &3.00e-05  &2.47    &5.38e-04  &2.00\\
64          &4.38e-04  &1.33    &3.13e-06  &2.08    &6.51e-06  &2.21    &1.35e-04  &2.00\\
128         &2.03e-04  &1.11    &7.70e-07  &2.02    &1.56e-06  &2.06    &3.37e-05  &2.00\\
\hline
&\text{$P_3(T)-V_b(\pa T,2)-[P_2(T)]^2$~element~on~the~curved~uniform~meshes}&&&&&&&\\
\hline
4           &1.05e-02  &-       &7.18e-04  &-       &3.93e-04  &-       &1.04e-02  &-\\
8           &1.76e-03  &2.57    &7.19e-05  &3.32    &5.98e-05  &2.71    &1.70e-03  &2.61\\
16          &2.29e-04  &2.94    &4.67e-06  &3.94    &4.38e-06  &3.77    &2.18e-04  &2.96\\
32          &2.93e-05  &2.97    &3.02e-07  &3.95    &2.99e-07  &3.87    &2.77e-05  &2.98\\
64          &3.74e-06  &2.97    &2.03e-08  &3.89    &2.12e-08  &3.82    &3.49e-06  &2.99\\
\hline
&\text{$P_3(T)-P_2(\pa T)-[P_2(T)]^2$~element~on~the~straight~uniform~meshes}&&&&&&&\\
\hline
4           &1.03e-02  &-       &7.63e-04  &-       &5.50e-04  &-       &1.00e-02  &-\\
8           &3.12e-03  &1.72    &1.98e-04  &1.95    &2.01e-04  &1.45    &2.67e-03  &1.91\\
16          &1.15e-03  &1.44    &5.05e-05  &1.97    &5.63e-05  &1.83    &7.12e-04  &1.91\\
32          &5.17e-04  &1.15    &1.26e-05  &2.01    &1.49e-05  &1.92    &2.01e-04  &1.83\\
64          &2.52e-04  &1.04    &3.11e-06  &2.01    &3.83e-06  &1.96    &5.97e-05  &1.75\\
\hline
\end{tabular}
\end{table}

{\bf Test case 2 (circular domain)} Here is the configuration of the test: the exact solution is $u=-(x^2+y^2-1)$; the domain is an unit circle $\O=\{\{x,y\}:x^2+y^2\leq1\}$; the curved WG element $P_2(T)-V_b(\pa T,1)-[P_1(T)]^2$ is used; and the curved uniform meshes on levels 1 \& 2 are shown in Figure \ref{Pictu-example1-9}. The exact solution $u$ and WG solution $u_0$ are plotted in Figure \ref{Pictu-example1-2}. As we can see from Table \ref{NE:exampe-2}, the error of the WG appriximation  in various norms on curved uniform meshes achieves an optimal order of convergence, which  consists with our theory.
\begin{figure}[!htbp]
\centering
\includegraphics[height=0.36\textwidth,width=0.36\textwidth]{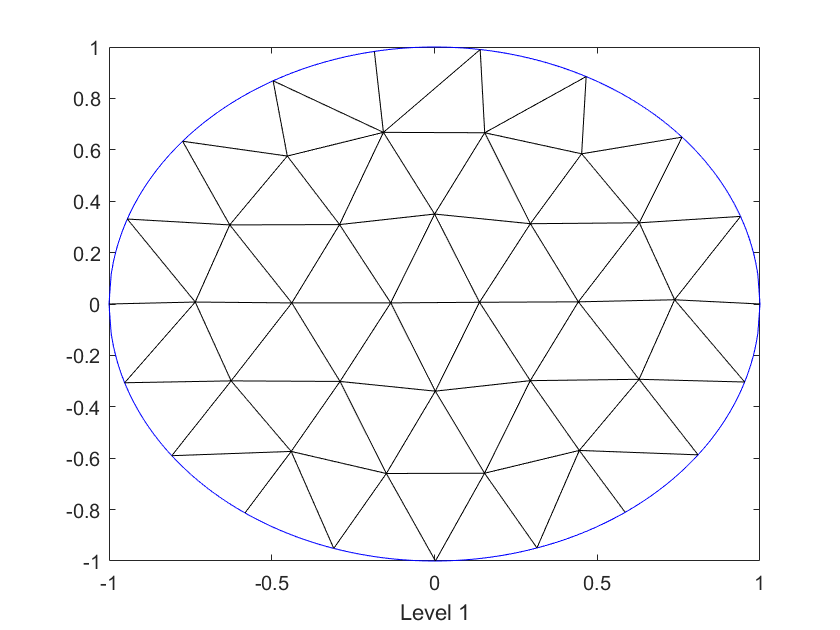}
\includegraphics[height=0.36\textwidth,width=0.36\textwidth]{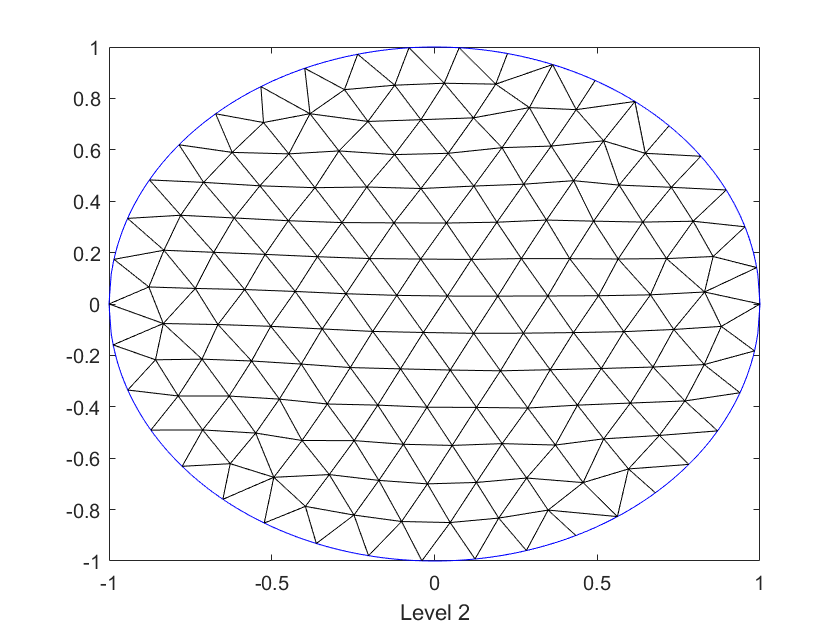}
\caption{Level 1 (Left) and level 2 (Right) of curved uniform meshes in test case 2.}\label{Pictu-example1-9}
\end{figure}

\begin{figure}
\centering
\includegraphics[height=0.395\textwidth,width=0.53\textwidth]{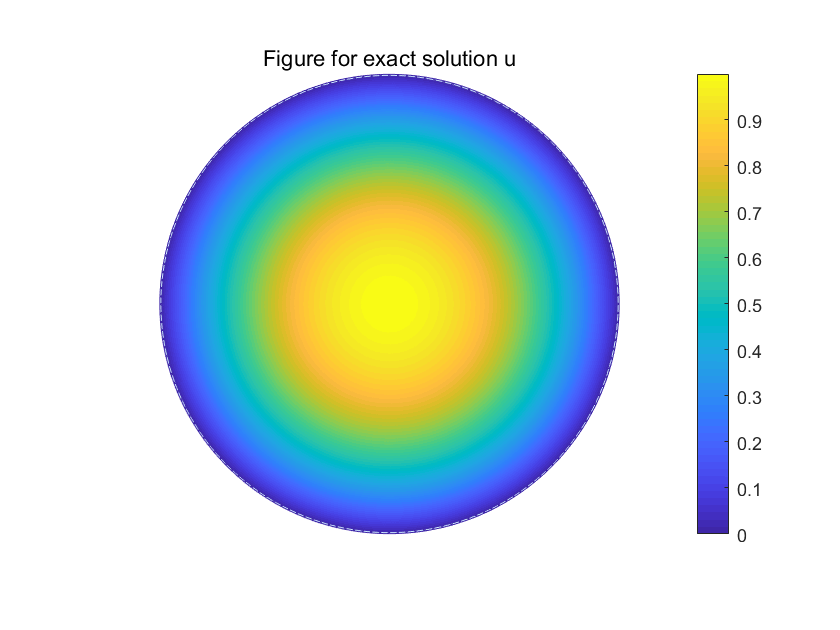}
\includegraphics[height=0.355\textwidth,width=0.42\textwidth]{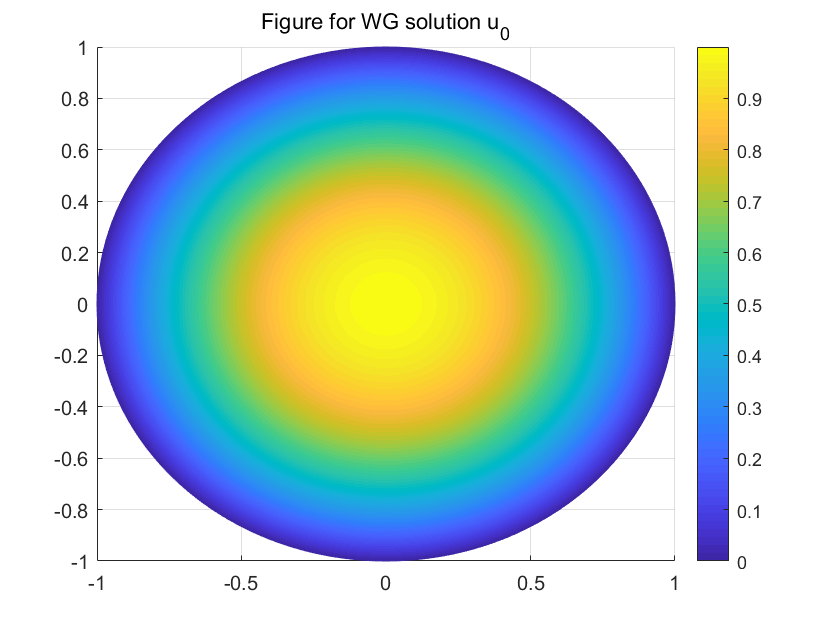}
\caption{$P_2(T)-V_b(\pa T,1)-[P_1(T)]^2$~element on the curved uniform meshes; Left: exact solution; Right: WG solution.}
\label{Pictu-example1-2}
\end{figure}

\begin{table}[h]
\caption{Test case 2: Numerical errors and corresponding convergence rates for $P_2(T)-V_b(\pa T,1)-[P_1(T)]^2$ element on the curved uniform meshes.}\label{NE:exampe-2}
\begin{tabular}{|p{0.9cm}p{1.28cm}p{0.8cm}p{1.28cm}p{0.8cm}p{1.28cm}p{0.8cm}p{1.28cm}p{0.6cm}|}
\hline
h&$\3bare_h\3bar$&Rate&$\|e_0\|$&Rate&$\|e_b\|_{\E_h}$&Rate&$\|\nabla e_0\|$&Rate\\
 \hline
0.3           &2.96e-04  &-       &3.47e-05  &-       &5.79e-05  &-       &2.40e-04  &-\\
0.15          &8.56e-05  &1.79    &4.43e-06  &2.97    &9.94e-06  &2.54    &6.27e-05  &1.94\\
0.075         &1.84e-05  &2.22    &4.97e-07  &3.16    &1.15e-06  &3.11    &1.33e-05  &2.23\\
0.0375        &4.50e-06  &2.03    &6.03e-08  &3.04    &1.46e-07  &2.97    &3.19e-06  &2.06\\
0.01875       &1.13e-06  &1.99    &7.20e-09  &3.07    &1.75e-08  &3.06    &6.99e-07  &2.19 \\
\hline
\end{tabular}
\end{table}

{\bf Test case 3 (circular disk)}
The configuration of the test is as follows: the domain is a circular disk defined by $\O=\{\{x,y\}:0.16\leq x^2+y^2\leq1\}$; the exact solution is given by $u=-(x^2+y^2-1)(x^2+y^2-0.16)$; the curved  WG element $P_2(T)-V_b(\pa T,1)-[P_1(T)]^2$ is used; and the curved uniform meshes on levels 1 \& 2 are shown in Figure \ref{Pictu-example1-10}. The plots of the exact solution $u$ and WG numerical approximation $u_0$ are demonstrated in Figure \ref{Pictu-example1-3}. We have observed from Table \ref{NE:exampe-3} that the error of WG solution in different norms on the curved uniform meshes achieves an optimal order of convergence. All numerical results are greatly consist with the theory established in this paper. 

\begin{figure}[!htbp]
\centering
\includegraphics[height=0.36\textwidth,width=0.36\textwidth]{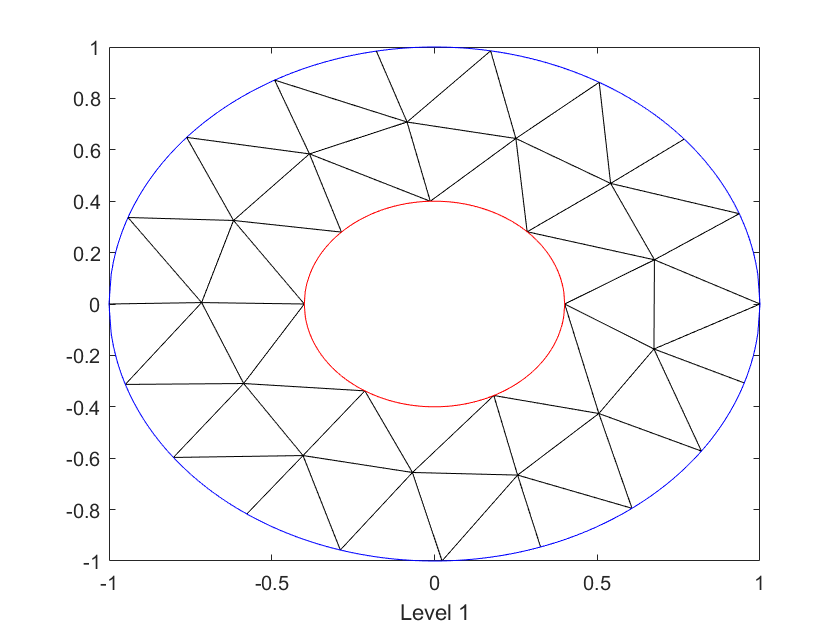}
\includegraphics[height=0.36\textwidth,width=0.36\textwidth]{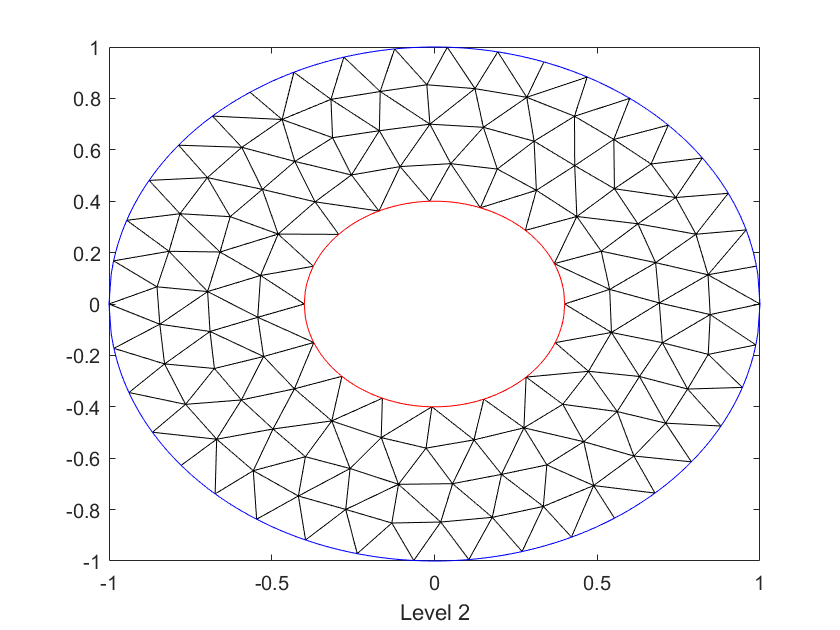}
\caption{Level 1 (Left) and level 2 (Right) of curved uniform meshes in test case 3.}\label{Pictu-example1-10}
\end{figure}

\begin{figure}[!htbp]
\centering
\includegraphics[height=0.37\textwidth,width=0.49\textwidth]{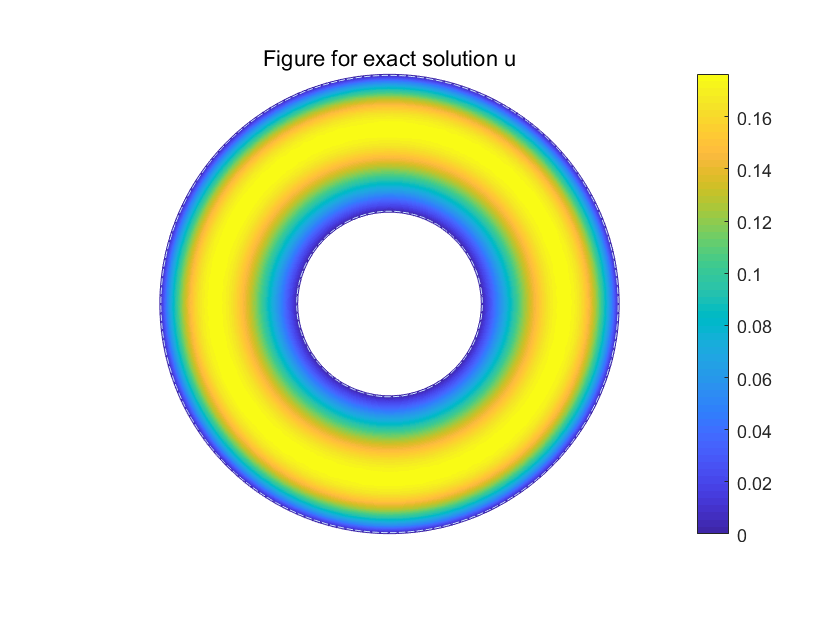}
\includegraphics[height=0.33\textwidth,width=0.39\textwidth]{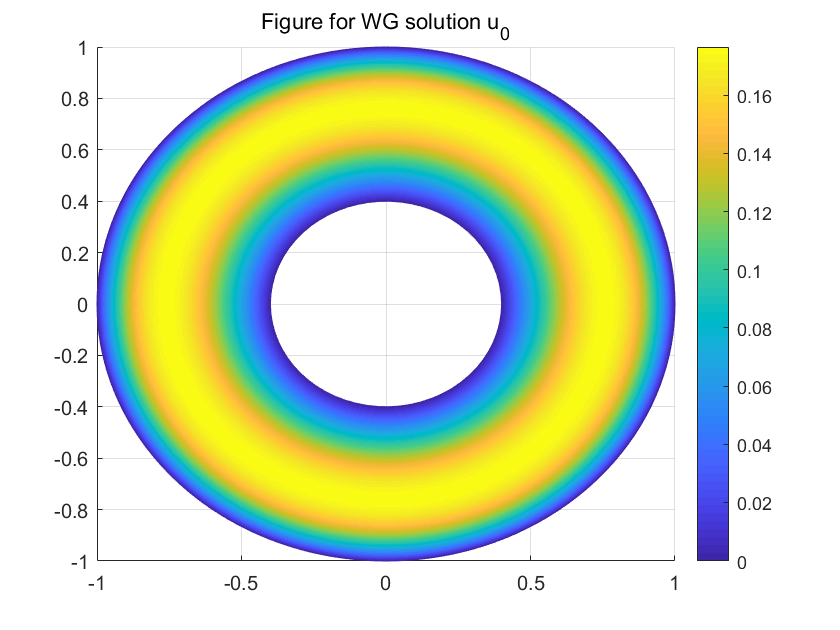}
\caption{$P_2(T)-V_b(\pa T,1)-[P_1(T)]^2$ element on the curved uniform meshes; Left: exact solution; Right: WG solution.}
\label{Pictu-example1-3}
\end{figure}

\begin{table}[h]
\caption{Test case 3: Numerical errors and corresponding convergence rates for $P_2(T)-V_b(\pa T,1)-[P_1(T)]^2$ element on the curved uniform meshes.}\label{NE:exampe-3}
\begin{tabular}{|p{0.9cm}p{1.28cm}p{0.8cm}p{1.28cm}p{0.8cm}p{1.28cm}p{0.8cm}p{1.28cm}p{0.6cm}|}
\hline
h&$\3bare_h\3bar$&Rate&$\|e_0\|$&Rate&$\|e_b\|_{\E_h}$&Rate&$\|\nabla e_0\|$&Rate\\
\hline
0.3           &8.43e-01  &-       &5.21e-02  &-       &3.80e-02  &-       &8.33e-01  &-\\
0.15          &2.01e-01  &2.07    &6.02e-03  &3.11    &4.15e-03  &3.19    &1.99e-01  &2.06\\
0.075         &5.00e-02  &2.01    &6.74e-04  &3.16    &4.48e-04  &3.21    &4.96e-02  &2.01\\
0.0375        &1.22e-02  &2.04    &7.79e-05  &3.11    &5.07e-05  &3.14    &1.21e-02  &2.04\\
0.01875       &3.02e-03  &2.01    &9.38e-06  &3.05    &6.04e-06  &3.07    &3.00e-03  &2.01 \\
\hline
\end{tabular}
\end{table}

\bigskip
\bigskip

\vfill\eject

\end{document}